\documentclass[a4paper,10pt]{article}
\usepackage{blindtext}
\usepackage{geometry}
\usepackage[utf8]{inputenc}
\usepackage[english]{babel}
\usepackage{amscd,amssymb,amsfonts,amsmath,amsthm}
\usepackage{bbm}
\usepackage{graphicx}
\usepackage{fancybox}
\usepackage{epic,eepic}
\usepackage{amstext}
\usepackage{mathtools}
\usepackage{hyperref}
\hypersetup{
    colorlinks=true,
    citecolor=red,
    linkcolor=blue,
    pdfborder={0 0 0}
}
\usepackage{esint}
\usepackage{subcaption}
\usepackage{enumitem}
\usepackage{color}
\usepackage{tikz}
\usepackage{pgfplots}
\usepackage{verbatim}
\usepackage{placeins}
\usepackage{soul}

\usepackage[noabbrev, capitalise, nameinlink]{cleveref}

\usepackage[resetlabels,labeled]{multibib}

\newtheorem{assumption}{Assumption}[section]
\newtheorem{theorem}{Theorem}[section]

\newtheorem{remark}{Remark}[section]

\newtheorem{proposition}{Proposition}[section]

\newcommand{\R}{\mathbb R}
\newcommand{\N}{\mathbb N}
\renewcommand{\P}{\mathbb P}
\newcommand{\B}{\mathbb B}
\newcommand{\T}{\mathcal T}

\newcommand{\Q}{\mathcal Q}

\newcommand{\Vb}{V_b}
\newcommand{\VbTh}{H^1(\T_h)}
\newcommand{\VbTL}{H^1(\T_L)}

\newcommand{\eremk}{\hbox{}\hfill\rule{1ex}{1ex}}

\newcommand{\norm}[2]{\left\lVert#1\right\rVert_{#2}}
\newcommand{\seminorm}[2]{\left |#1\right |_{#2}}

\newcommand{\vertiii}[1]{{\left\vert\kern-0.25ex\left\vert\kern-0.25ex\left\vert #1  \right\vert\kern-0.25ex\right\vert\kern-0.25ex\right\vert}}

\newcommand{\spE}[2]{\left(#1, #2\right)_{L^2(E)}}

\newcommand\ds{\,\text{d}s}

\newcommand\dx{\,\text{d}x}

\newcommand\mus{{\mu_s}}
\newcommand\mut{{\mu_t}}

\newcommand{\E}{\mathbb E}

\newcommand{\dP}{\mathrm{d}\mathbb P(\omega)}
\newcommand{\mc}[1]{E_{#1}}
\newcommand{\mlmc}{E^{L}}

\newcommand{\pa}[1]{\textcolor{blue}{#1}}
\newcommand{\todo}[1]{\textcolor{black}{#1}} 

\begin{document}

\title{A Multilevel Monte Carlo Virtual Element Method for Uncertainty Quantification of Elliptic Partial Differential Equations\thanks{This research has been partially funded by the European Union (ERC, NEMESIS, project number 101115663). Views and opinions expressed are, however, those of the author(s) only and do not necessarily reflect those of the European Union or the European Research Council Executive Agency. Neither the European Union nor the granting authority can be held responsible for them. This research was also funded in part by the Austrian Science Fund (FWF) 10.55776/F65. The present research is part of the activities of “Dipartimento di Eccellenza 2023-2027”. PA and MV also acknowledge MUR–PRIN/PNRR 2022 grant n. P2022BH5CB, funded by MUR. 
FB also acknowledges INdAM - GNCS Project CUP E53C24001950001.
PFA, FB and MV are members of INdAM-GNCS. }}

\author{Paola F. Antonietti$^a$, Francesca Bonizzoni$^a$, Ilaria Perugia$^b$, Marco Verani$^a$}

\date{}

\maketitle

\begin{center}
{\small $^a$ MOX, Dipartimento di Matematica, Politecnico di Milano, Milano, Italy\\
$^b$ Faculty of Mathematics, University of Vienna, Vienna, Austria}
\end{center}
\begin{center}
\today\\
\end{center}

\begin{abstract}
We introduce a Monte Carlo Virtual Element estimator based on Virtual Element discretizations for stochastic elliptic partial differential equations with random diffusion coefficients. We prove estimates for the statistical approximation error for both the solution and suitable linear quantities of interest. A Multilevel Monte Carlo Virtual Element method is also developed and analyzed to mitigate the computational cost of the plain Monte Carlo strategy. The proposed approach exploits the flexibility of the Virtual Element method on general polytopal meshes and employs sequences of coarser spaces constructed via mesh agglomeration, providing a practical realization of the multilevel hierarchy even in complex geometries. This strategy substantially reduces the number of samples required on the finest level to achieve a prescribed accuracy. We prove convergence of the multilevel method and analyze its computational complexity, showing that it yields significant cost reductions compared to standard Monte Carlo methods for a prescribed accuracy. Extensive numerical experiments support the theoretical results and demonstrate the efficiency of the proposed method.\\[0.5cm]
\end{abstract}



\section{Introduction}

Differential models with uncertain data are common in engineering, the social sciences, and the life sciences. Indeed, models’ input data, such as coefficients or the source term, are often uncertain due to, e.g., intrinsic variability or measurement errors. A typical approach to handling such uncertainties is to model the uncertain data as random fields and analyze statistical quantities associated with solutions to the resulting stochastic PDEs. These quantities often include the expected value of the solution or functionals (quantity of interest, QoI) that depend on it. The Monte Carlo (MC) method is widely used to approximate solutions to stochastic PDEs by solving the corresponding deterministic problem repeatedly with independent realizations of the random input and computing statistical measures. However, this method can be impractical when each deterministic problem is computationally expensive, since the error decreases at a rate of $O(M^{-1/2})$, where $M$ is the number of samples. 
To overcome this limitation, Multilevel Monte Carlo (MLMC) methods have been developed. The key idea is to combine many inexpensive coarse simulations with fewer costly fine ones, achieving an improved complexity--error balance.
Our approach builds on the MLMC framework for
infinite-dimensional integration introduced by Giles in the context of stochastic differential equations~\cite{Giles2008, giles2008multilevel}. It was then extended to elliptic PDEs with random coefficients in~\cite{BarthSchwabZollinger2011, Cliffe-Giles-Scheichl-Teckentrup, CharrierScheichlTeckentrup2013, TeckentrupScheichlGilesUllmann} and more recently in~\cite{BeckLiuVonSchwerinTempone2022, SchwabStein2024}.
For a review of MLMC and 
its variants, see~\cite{Giles_2015}.
The effectiveness of MLMC depends on the availability of a sequence of coarser discretizations to approximate the corresponding deterministic problem. For example, in the Finite Element method, a sequence of coarser meshes is required for practical MLMC implementation. In complex geometries, constructing such a sequence of coarser (simplicial or hexahedral) meshes can be challenging. On the other hand, mesh agglomeration, which merges groups of fine-mesh elements into larger cells of arbitrary polygonal or polyhedral shape, provides a robust, automatic way to generate necessary coarser problems \cite{ANTONIETTI2024,antonietti2025magnet, FEDER2025,Antonietti2026}. This has driven the development of discretization techniques capable of handling general polytopal meshes. Several polytopal methods that have been developed in the last decades are based on nonconforming approximation spaces, such as Discontinuous Galerkin (DG) methods~\cite{DiPietro-Ern,CangianiDongGeorgoulisHouston_2017} or Hybrid High-Order (HHO) methods~\cite{Cicuttin-Ern-Pignet,DiPietro-Droniou}. In contrast, the Virtual Element (VE) method relies on conforming approximation spaces. In
this paper, we focus on the 
VE
method that has been proposed in the seminal papers \cite{beirao2013basic,BBM_2013,beirao2014hitchhiker} and further successfully developed for a wide set of problems in computational sciences; see, e.g., the monographs~\cite{BBM_acta_2023} and~\cite{Antonietti2022VEM}, and the references therein. \\

In this work, we consider a stochastic elliptic PDE with random diffusion coefficients and aim to approximate 
both the expected value of the solution and that of a linear quantity of interest. We introduce and analyse a novel multilevel Monte Carlo Virtual Element (MLMC-VE) method for this purpose. VE methods allow, simultaneously, the natural construction of multilevel discretizations on general polytopal meshes required in the multilevel framework, and guarantee stability and approximation properties across levels. 
 Given the virtual nature of the basis functions, the design and analysis 
of MLMC-VE is far from being a straightforward extension of the FEM case, as it requires careful use of the projectors of the virtual functions in the polynomial space.

We first derive novel VE a priori error estimates for linear quantities of interest of the solution to deterministic elliptic problems. Building on this result, we present Monte Carlo Virtual Element (MC-VE) estimators for the expectation of linear quantities of interest of the solution to stochastic elliptic problems with a random diffusion coefficient. We prove error bounds for the MC-VE estimator, showing that
as expected, the MC convergence rate is $1/2$, so to halve the error, one should quadruple the sample size $M$. Given that, for each sample, the MC-VE method requires the solution of a VE problem, the overall computational complexity of the MC-VE method might become very large even for not too small error tolerances, as in the classical Monte Carlo setting.
Next, to reduce the computational complexity of the MC-VE method, we propose a multilevel Monte Carlo Virtual Element (MLMC-VE) that uses a hierarchy of VE discretizations. The MLMC-VE estimator significantly reduces the sample count on the finest meshes, thereby improving efficiency. Exploiting the linearity of the expectation, the method computes the difference between expectations of the solutions (or their QoI) at two consecutive levels of accuracy. The advantage lies in reducing problem variance, thereby decreasing the number of samples required at finer (and more expensive) levels. We prove convergence, discuss computational complexity, and show that it significantly decreases the computational cost to achieve a given accuracy compared to standard Monte Carlo methods.
Numerical experiments confirm the 
theoretical error estimates and show the
computational benefits of our approach. 
To the best of our knowledge, this is one of the first results that incorporate stochastic modeling techniques into the VE framework to address uncertainty and to develop stochastic virtual element methods, in which deterministic VE spaces are coupled with stochastic approximation spaces to represent random variables and stochastic fields.
In this direction, we also mention the recent paper~\cite{zheng_stochastic_2024}, which proposes stochastic VE methods for two- and three-dimensional linear elasticity problems.\\

The rest of the paper is organized as follows. In Section~\ref{sec:problem}, we introduce the model problem, namely an elliptic PDE with stochastic diffusion coefficient, together with the functional setting. In Section~\ref{sec:VEM}, we provide a brief overview of the conforming Virtual Element discretization of the corresponding deterministic model problem, and in
Section~\ref{sec:qoi_apriori}, we derive novel \emph{a priori} error estimates for a linear QoI. In Section~\ref{sec:mcvem}, we introduce the MC-VE method and derive estimates for the statistical approximation error for both the solution and the linear quantity of interest. The extension to the multilevel setting is addressed in Section~\ref{sec:MLMC-VE}, where we introduce and establish error estimates for the MLMC-VE estimators, for both the solution and the quantity of interest. Extensive numerical experiments are presented and discussed in Section~\ref{sec:numerics}. Finally, in Section~\ref{sec:conclusions}, we draw some conclusions and discuss possible future extensions.\\

\section{Problem setting}\label{sec:problem} 
In this section, we introduce the model problem considered in this work, namely, an elliptic partial differential equation with a stochastic diffusion coefficient. \\ 

Let $(\Omega, \mathcal F,\mathbb P)$ be a probability space, where $\Omega$ denotes the set of outcomes, $\mathcal F$ the $\sigma$-algebra of events, and $\mathbb P\colon\Omega\rightarrow[0,1]$ a probability measure. Moreover, let $D\subset \R^2$ be an open, bounded, convex polygonal domain with 
boundary $\partial D$. 
In this work, we focus on the following elliptic PDE with stochastic diffusion coefficient: given~$f\in L^2(D)$ and $a\colon\Omega\times D\rightarrow\R$, find $u\colon\Omega\times D\rightarrow\R$ such that, almost surely in $\Omega$,
\begin{gather}
\left\{\begin{array}{ll}
     -\nabla\cdot (a\nabla u)=f & \text{in } D, \\
     u=0& \text{on } \partial D. 
\end{array}\right.
    \label{eq:strong_pde}
\end{gather}

Given the realization $a(\omega,\cdot)$, the weak formulation of \eqref{eq:strong_pde} reads:  find $u(\omega,\cdot)\in V= H^1_0(D)$ such that, for all $v\in V$,
\begin{gather}
    \int_{D} a(\omega,\cdot)\nabla u(\omega,\cdot)\cdot \nabla v\, \dx
    =\int_{D} f\, v\, \dx,
    \label{eq:weak_pde}
\end{gather}
which is well-posed by the Lax-Milgram lemma, provided that $a(\omega,\cdot)$ is uniformly bounded below and above in $D$ by positive constants.
We make the following assumption.
\begin{assumption}
\label{ass:a}
   For almost all $\omega\in\Omega$, the realizations $a(\omega,\cdot)$ of the diffusion coefficient belong to $L^\infty(D)$ and satisfy
   \begin{gather*}
        0<a_{min}(\omega)\leq a(\omega,x)\leq a_{max}(\omega)<\infty\quad\mathrm{a.e.\ in\ }D,    
   \end{gather*}
   where 
   \[
   a_{min}(\omega)= \mathrm{ess\,inf}_{x\in D}a(\omega,x),\quad
   a_{max}(\omega)= \mathrm{ess\,sup}_{x\in D}a(\omega,x).
   \]
\end{assumption}
Note, in particular, that a random field $a(\omega,x)$ satisfying Assumption~\ref{ass:a} may not be uniformly bounded over $\Omega$.
Under \pa{A}ssumption~\ref{ass:a}, repeated application of the Lax-Milgram lemma yields the almost surely well-posedness of problem~\eqref{eq:strong_pde}, namely, for almost every~$\omega\in\Omega$, problem~\eqref{eq:weak_pde} admits a unique solution $u(\omega,\cdot)\in V$ and there exists a positive constant $C$ independent of $\omega$ such that
\[
\|u(\omega,\cdot)\|_{V}
\leq C a_{min}(\omega)^{-1}\|f\|_{L^2(D)}. 
\]
\begin{assumption}
\label{ass:a2}
  For any $p\geq 1$, $a_{min}^{-1}\in L^p(\Omega)$ and $a\in L^p(\Omega, C^0(\overline{D}))$.
\end{assumption}
Under \pa{A}ssumption~\ref{ass:a2}, it can be proved that $u\in L^p(\Omega,V)$ for any $p\geq 1$ and, for a constant $C>0$,
\begin{gather}
\label{eq:bound_u}
\|u\|_{L^p(\Omega,V)}\leq C \|a_{min}^{-1}\|_{L^p(\Omega)}\|f\|_{L^2(D)},
\end{gather}
where the Bochner norm is defined by
\[
\norm{\cdot}{L^p(\Omega,V)}= 
\left(\int_\Omega | \cdot|_{H^1(D)}^p \right)^{1/p}.
\]
Note that Assumption~\ref{ass:a2} holds, for example, for random fields~$a(\omega,x)$ that are uniformly bounded in $\Omega\times D$, as well as for log-normal random fields $a(\omega,x)=\mathrm{e}^{Y(\omega,x)}$ with $Y\colon\Omega\times\mathbb{R}$ mean-zero Gaussian random field fulfilling, for some $L,s>0$ and for all $x_1,\, x_2\in \overline{D}$,
\[
\E[|Y(\omega,x_1)-Y(\omega,x_2)|^2]\leq L \|x_1-x_2\|_2^s.
\]

\begin{remark}
For simplicity, we have assumed that the forcing term is deterministic. Nevertheless, the results mentioned above, as well as the forthcoming analysis, extend to the more general case where $f=f(\omega,x)$ is a sufficiently regular random field (namely, $f\in L^p(\Omega,L^2(D))$ for any $p>1$).
More general boundary conditions, other than homogeneous Dirichlet, can be handled with minor modifications. 
Moreover, the proposed methodology naturally extends to the three-dimensional case $D\subset\R^3$.
\eremk
\end{remark}


 The present work aims to approximate the expected value of the solution $u$ to problem~\eqref{eq:strong_pde}, i.e., $$\E[u]=\int_\Omega u(\omega,\cdot)\,\dP\in V$$
and the expected value of a quantity of interest $\mathcal Q\colon V\rightarrow\R$ of $u$, i.e., $$\E[\mathcal Q(u)]=\int_\Omega \mathcal Q(u(\omega,\cdot))\,\dP\in \R,$$ where we assume that $\Q:V\to \R$ is a linear operator defined, for a given $q\in L^2(D)$, as 
\begin{equation}\label{eq:q}
\Q(v)=\int_D q v.
\end{equation}

\section{The Virtual Element Method}\label{sec:VEM}

In this section, we provide a brief overview of the conforming virtual element~(VE) method and the main theoretical results that form the basis for the design and study of the Monte Carlo and Multilevel Monte Carlo VE discretizations introduced in this paper. Here and in the following, we will make use of the notation $\lesssim$ to denote a bound that holds up to a positive multiplicative constant $C$ that may vary at different occurrences, but it is independent of the mesh size, the polynomial degree, 
and the model's parameters.\\

We consider 
the following elliptic PDE with a \emph{deterministic} diffusion coefficient: given $f\colon D\rightarrow\R$ and $\alpha\colon D\rightarrow\R$, find $u\colon D\rightarrow\R$ such that
\begin{gather}
\left\{\begin{array}{ll}
     {-\nabla\cdot (\alpha\nabla u)=f} & {\text{in } D,}\\
     {u=0} & {\text{on } \partial D,}
\end{array}\right.\label{eq:strong_pde_det}
\end{gather} 
where we assume $f\in L^2(D)$ and  $\alpha$ uniformly bounded from above and below in $D$ by positive
constants being
\begin{gather}
    \label{eq:alpha_bounds}
   \alpha_{min}= \mathrm{ess\,inf}_{x\in D}\alpha(x),\quad
   \alpha_{max}= \mathrm{ess\,sup}_{x\in D} \alpha(x).
\end{gather}
The weak formulation of \eqref{eq:strong_pde_det}
reads as follows: find $u \in H^1_0(D)$ such that 
\begin{equation}\label{eq:cont_pbl}
    A(u,v)=F(v) \qquad\forall v\in H^1_0(D),
\end{equation}
where 
\begin{eqnarray}
    &&A(u,v):H^1_0(D) \times H^1_0(D) \to \R, \qquad A(u,v)=\int_D \alpha\, \nabla u \cdot \nabla v, \label{A:def}\\
    && F(v):H^1_0(D)\to \R, \qquad F(v)=\int_D fv.\label{F:def}
\end{eqnarray}

The remainder of this section is devoted to introduce the VE spaces, the bilinear forms, and the corresponding discrete problem.
\subsection{Virtual Element spaces}
\label{sec:VE_space}
Let $\T_h$ be a decomposition of $D$ 
into $n_P$ non-overlapping (open) polygons $E_\ell$ with flat faces, i.e., $\overline{D} 
=\cup_{\ell=1}^{n_P} \bar E_\ell$ with $E_\ell\cap E_{\ell'}=\emptyset$ for $\ell\neq\ell'$. Let $h_E=\operatorname{diam}(E)$ denote the diameter of the element $E$, and $h=\max_{E\in\T_h}h_E$ denote the meshsize. We assume that there exist constants $\gamma, c>0$ independent of $h$ and $E$ such that
\begin{itemize}
\item each element $E\in\T_h$ is star-shaped with respect to a ball of radius $\gamma\, h_E$, 
\item the distance between any two vertices of $E$ is larger than $c\, h_E$. 
\end{itemize}
We denote by~$\VbTh$ the broken $H^1$ space on $\T_h$: 
$$\VbTh=\{v\in L^2(D): v\vert_E\in H^1(E)\quad \forall E\in \T_h\}$$ 
and by~$\seminorm{\cdot}{\VbTh}^2=\sum_{E\in\T_h}\seminorm{\cdot}{H^1(E)}^2$ the broken $H^1$ seminorm.
Let $p\in\N$ be fixed.
For any 
$E\in\T_h$, we introduce the following notation:
\begin{enumerate}[label=(\roman*)]
\item
$\P_p(E)$ is the set of polynomials on $E$ of total degree less than or equal to~$p$;
\item
$\B(\partial E)=\{v\in C^0(\partial E)\ \textrm{s.t. }v|_e\in\P_p(e)\ \textrm{for all edges }e\subset\partial E\}$;
\item 
$\Pi_{p}^{\nabla,E}\colon H^1(E)\rightarrow \P_p(E)$ is the energy projection operator, which is defined, up to a constant, by
\begin{equation}
	\label{eq:PiNabla}
    \int_E \nabla(w-\Pi_{p}^{\nabla,E}w)\cdot \nabla q_p\dx=0
    \quad\forall\, q_p\in\P_p(E).
\end{equation}
To fix the constant and complete the definition 
of $\Pi_{p}^{\nabla,E}w$, we further require 
\begin{gather*}
{
\left\{\begin{array}{ll}
\displaystyle{\frac{1}{|\partial E|}\int_{\partial E} (w-\Pi_{p}^{\nabla,E}w)\, \ds =0}, & \text{ if }p=1,\\
\displaystyle{\frac{1}{|E|}\int_E (w-\Pi_{p}^{\nabla,E}w)\, \dx =0}, & \text{ if }p>1;
\end{array}\right.}
\end{gather*}
\item
$\Pi_p^{0,E}\colon L^2(E)\rightarrow \P_p(E)$ is the $L^2$-orthogonal projection operator defined by
\begin{equation}
	\label{eq:Pi0}
	\spE{q_p}{w-\Pi_{p}^{0,E}w}=0\quad\forall\, q_p\in\P_p(E).
\end{equation}
\end{enumerate}
Since~$\Pi_{p}^{\nabla,E}$ and~$\Pi_{p}^{0,E}$ provide the best approximation in~$\P_p(E)$ with respect to the~$H^1(E)$ seminorm and in the~$L^2(E)$ norm, respectively, from~\cite[Lemma~4.2]{BCMR:2016}, it follows that,
for all $u\in H^{s+1}(E)$, $s\ge 1$,
letting~$\mus=\min\{s,p\}$, 
\begin{align}
\label{eq:proj_approx_nabla_E}
    \seminorm{u-\Pi_{p}^{\nabla,E} u}{H^1(E)}
    &\lesssim 
    \frac{h_E^{\mus}}{p^s} 
    \norm{u}{H^{s+1}(E)},\\
    \label{eq:proj_approx_E}
    \norm{u-\Pi_{p}^{0,E} u}{L^2(E)}
    &\lesssim 
    \frac{h_E^{\mus+1}}{p^{s+1}}
    \norm{u}{H^{s+1}(E)}.
\end{align}
Moreover, the global operators $\Pi^\nabla_{h,p}$, $\Pi^0_{h,p}$ are defined as $\Pi^{(\cdot)}_{h,p}\vert_E=\Pi^{(\cdot),E}_{h,p}$ with $(\cdot)=\nabla,0$ and, for all $u\in H^{s+1}(D)$, $s\ge 1$, 
fulfill
\begin{align}
\label{eq:proj_approx_nabla}
    \seminorm{u-\Pi_{h,p}^{\nabla} u}{\VbTh}
    &\lesssim 
    \frac{h^{\mus}}{p^s} 
    \norm{u}{H^{s+1}(D)},
    \\
    \label{eq:proj_approx}
    \norm{u-\Pi_{h,p}^{0} u}{L^2(D)}
    &\lesssim 
    \frac{h^{\mus+1}}{p^{s+1}}
    \norm{u}{H^{s+1}(D)}.
\end{align}
Additionally, for all $u\in H^1(D)$, we have
\begin{equation}\label{eq:L2H1}
\norm{u-\Pi_{h,p}^{0} u}{L^2(D)}\lesssim \frac{h}{p} \seminorm{u}{H^{1}(D)}.
\end{equation}
Finally, from~\eqref{eq:PiNabla}, we have the following local and global stability properties:
\begin{gather}
    \label{eq:proj_stability_nabla}
    \seminorm{\Pi_{p}^{\nabla,E} u}{H^1(E)}
    \le \seminorm{u}{H^{1}(E)}, \qquad 
    \seminorm{\Pi_{h,p}^{\nabla} u}{\VbTh}
    \le  \seminorm{u}{H^{1}(D)}.
\end{gather}

We can now introduce the local enhanced VE space of order 
$p\geq 1$:
\begin{equation}
	\label{eq:VEM_space_local}
	W^E_{p}=\left\{ w\in V^E_{p}\ \textrm{s.t. }\spE{w-\Pi_{p}^{\nabla,E}w}{q}=0\ \textrm{for all }q\in\P_p(E)/\P_{p-2}(E)\right\},
\end{equation}
where $V^E_{p}$ denotes the local augmented VE space
\[
V^E_{p}=\left\{w\in H^1(E)\ \textrm{s.t. }w\in\B_p(\partial E)\ \textrm{and }\Delta w\in\P_p(E)\right\},
\]
and $\P_p(E)/\P_{p-2}(E)$ denotes the set of polynomials of total degree $p$ on $E$ that are $L^2$-orthogonal to all polynomials of total degree $p-2$ on $E$ (with the convention $\P_{-1}= \emptyset$).  Note, in particular, that $\P_p(E)\subset W^E_{h,p}(E)$.
The space $W^E_{h,p}$ is equipped with the following set of (local) degrees of freedom (DOFs):
\begin{itemize}
\item
nodal values at all $n_E$ vertices of the polygon $E$;
\item
nodal values at the $p-1$ {internal} Gauss-Lobatto quadrature points of each edge $e\in\partial E$;
\item 
(for $p\geq 2$) moments up to order $p-2$ in $E$, namely, for $w\in W^E_{p}$,
\[
\spE{w}{q_{p-2}}\quad\textrm{for all }q_{p-2}\in\P_{p-2}.
\]
\end{itemize}
We have that $\dim(W^E_{p})=n_E p + \frac{p(p-1)}{2}$. It is important to note that both the energy and the $L^2$-orthogonal projection operators are computable solely in terms of the DOFs defined above.
The global enhanced VE space of degree $p$ is defined by
\begin{equation}
	\label{eq:VEM_space_global}
	W_{h,p}=\left\{v\in H^1_0(D)\ \textrm{s.t. }v|_E\in W^E_{p}\ \textrm{for all } E\in\T_h \right\}.
\end{equation}
It is equipped with the following set of (global) DOFs:
\begin{itemize}
	\item
	nodal values at all $n_V$ vertices of $\T_h$;
	\item
	nodal values at the $p-1$ {internal} Gauss-Lobatto quadrature points of each of the $n_e$ edges of $\T_h$;
	\item (for $p\geq 2$) moments up to order $p-2$ in each of the $n_P$ polygons of $\T_h$, namely, for $w\in W_{h,p}$,
	\[
	\spE{w}{q_{p-2}}\quad\textrm{for all }q_{p-2}\in \P_{p-2},\ \textrm{for all } E\in \T_h,
	\]
\end{itemize}
and it has dimension {$N_{h,p}=\dim(W_{h,p})=n_V + (p-1)n_e+n_P\frac{p(p-1)}{2}$}. 
It was shown in~\cite[Corollary~2.3]{CGMMV:2020} that, for any 
$w\in H^{s+1}(D)$, $s\ge 1$, there exists a VE function~$w_I\in W_{h,p}$ such that
\begin{equation}
	\label{eq:inter-estimate}
    \seminorm{w-w_I}{H^1(D)}
	\lesssim \frac{h^{\mus}}{p^{s-1}}\norm{w}{H^{s+1}(D)},
\end{equation}
where 
we recall the definition~$\mus=\min\{s,p\}$. 
\begin{remark}
We point out that, although the above result \eqref{eq:inter-estimate} is expected to be valid with~$p^s$ instead of~$p^{s-1}$, in the following, we use~\eqref{eq:inter-estimate}, which is the current state of the art.
\eremk
\end{remark}

\subsection{Virtual Element bilinear forms}

Let $A^E(u,v)=\int_E \alpha(x) \nabla u \cdot \nabla v \dx$.
Given an arbitrary pair of VE functions $v_{h,p},\, w_{h,p}\in W^E_{p}$, the quantity $A^E(v_{h,p},w_{h,p})=\int_E \alpha\, \nabla v_{h,p} \cdot \nabla w_{h,p}$ is not computable in terms of the VE DOFs. Thus, we introduce the computable approximation $A^E_h\colon W^E_{p}\times W^E_{p}\rightarrow\R$, given by
\begin{gather}
	\label{eq:VEM_bilinear_forms}
	\begin{aligned}
		A^E_h(v_{h,p},w_{h,p})&= A^E(\Pi_p^{\nabla,E}v_{h,p},\Pi_p^{\nabla,E}w_{h,p})
			+  S^E\left((Id-\Pi_p^{\nabla,E})v_h,(Id-\Pi_p^{\nabla,E})w_h \right),
	\end{aligned}
\end{gather}
 where~$S^E\colon W^E_{p}\times W^E_{p}\rightarrow\R$ denotes a computable, symmetric, stabilizing bilinear form satisfying, for all $w_h\in W^E_{p}$ with $\Pi_p^{\nabla,E}w_{h,p}=0$,
\begin{gather}
	\begin{aligned} \label{eq:equiv}
		s_\star(p) A^E(w_{h,p},w_{h,p})\lesssim 
        S^E(w_{h,p},w_{h,p}) \lesssim s^\star(p)A^E(w_{h,p},w_{h,p})
	\end{aligned}
\end{gather} 
for some positive constants~$s_\star(p),s^\star(p)$, with~$s^\star(p)\ge 1$, which may depend only on~$p$.

\begin{remark}
In our numerical experiments, we 
define 
the stabilizing form $S^E$ 
as follows:
\begin{gather*}
	\begin{aligned}
S^E (v_{h,p},w_{h,p})&=
	\alpha_E\sum_{r=1}^{N_E} \textrm{DOF}_r\Big( (\textrm{Id}-\Pi_p^{\nabla,E})v_{h,p} \Big)
	\textrm{DOF}_r\Big( (\textrm{Id}-\Pi_p^{\nabla,E})w_{h,p} \Big),\\
	\end{aligned}
\end{gather*}
where~$\alpha_E=\frac{1}{\vert E\vert }\int_E \alpha \dx$,$N_E=\dim(W^E_{h,p})$, and $\{\textrm{DOF}_r\}_{r=1}^{N_E}$ denotes the set of local DOFs introduced in Section~\ref{sec:VE_space}. For~$\alpha\in W^{1,\infty}(D)$ and sufficiently fine meshes, $S^E$ satisfies~\eqref{eq:equiv}.
\eremk
\end{remark}

The global VE bilinear form $A_h\colon W_{h,p}\times W_{h,p}\rightarrow\R$ is then defined, for all $v_{h,p},\, w_{h,p}\in W_{h,p}$, as
\begin{gather*}
	A_h(v_{h,p},w_{h,p})=\sum_{E\in\T_h} A^E_h(v_{h,p},w_{h,p}).
\end{gather*}
The global VE bilinear form is continuous and coercive: for all~$v_{h,p},\, w_{h,p}\in W_{h,p}$, 
\begin{gather*}
	\label{eq:cont_VEM_forms}
	\begin{aligned}
	A_h(v_{h,p},w_{h,p})&\lesssim s^\star(p)\alpha_{max}\seminorm{v_{h,p}}{H^1(D)}\seminorm{w_{h,p}}{H^1(D)},\\
    A_h(v_{h,p},v_{h,p})&\gtrsim s_\star(p)\alpha_{min}\seminorm{v_{h,p}}{H^1(D)}^2.
	\end{aligned}
\end{gather*}

\subsection{The discrete problem}\label{sec:discreteproblem}

The VE discretization of problem~\eqref{eq:strong_pde_det} reads as follows: find $u_{h,p}\in W_{h,p}$ such that 
\begin{equation}\label{pb: vem}
    A_h(u_{h,p},v_{h,p})=F_h(v_{h,p}) \qquad \forall v_{h,p}\in W_{h,p},
\end{equation}
where 
$$
F_h(v_{h,p})
=\sum_{E\in \T_h}\int_E \Big(\Pi^{0,E}_p f\Big)\, v_{h,p} 
=\sum_{E\in \T_h}\int_E  \Big(\Pi^{0,E}_p f\Big)\, \Big(\Pi^{0,E}_p v_{h,p}\Big). 
$$  
From the stability of the~$L^2$ projection in the~$L^2$ norm and the Poincar\'e inequality, for all~$v_{h,p}\in W_{h,p}$, we have
\[
|F_h(v_{h,p})|\lesssim \norm{f}{L^2(D)}\seminorm{v_{h,p}}{H^1(D)},
\]
which, together with the coercivity of~$A_h$, implies
\begin{equation}\label{eq:H1}
\seminorm{u_{h,p}}{H^1(D)}\lesssim s_\star(p)^{-1}\alpha_{min}^{-1}\norm{f}{L^2(D)},
\end{equation}
Then, the discrete problem~\eqref{pb: vem} is well-posed. Moreover, if $u\in H^{s+1}(D)$, $s\ge 1$, 
the following error estimates has been proven, e.g., in~\cite[Theorem~4.1, Remark~4.4, and Remark~4.5]{BCMR:2016}: 
\begin{equation}\label{vem:estimate}
\begin{split}
    |u-u_{h,p}|_{H^1(D)} & \lesssim \frac{s^\star(p)}{s_\star(p)}\frac{\alpha_{max}}{\alpha_{min}}
    \frac{h^{\mus}}{p^s}
    \|u\|_{H^{s+1}(D)},\\
    \|u-u_{h,p}\|_{L^2(D)} &\lesssim \frac{s^\star(p)}{s_\star(p)}\frac{\alpha_{max}}{\alpha_{min}}
    \frac{h^{\mus+1}}{p^{s+1}}
    \|u\|_{H^{s+1}(D)},
    \end{split}
\end{equation}
with~$\mus=\min\{s,p\}.$

\section{A priori Virtual Element error estimates for a linear quantity of interest}
\label{sec:qoi_apriori}

In this section, we present novel \emph{a priori} VE error estimates for a linear QoI of the solution in the deterministic setting.\\

Recalling that $V=H^1_0(D)$, and $\Q:V\to \R$ is the linear operator defined in~\eqref{eq:q}, we are interested in estimating $\Q (u)-\Q(u_{h,p})$, 
where $u\in V$ is the solution of~\eqref{eq:cont_pbl},
and $u_{h,p}\in W_{h,p}$ is the VE approximation to $u$, i.e., it satisfies~\eqref{pb: vem}.
For~$v_{h,p}\in W_{h,p}$, let us define 
    \begin{equation}\label{QoI:discrete}
    \Q_h(v_{h,p})=\int_D (\Pi^0_{p,h} q) \, v_{h,p},  
    \end{equation}
 which, unlike~$\Q(u_{h,p})$, is a computable quantity, since it holds
  \begin{equation}
    \label{eq:Qh_identity}
     \Q_h(v_{h,p}) = \int_D  (\Pi^0_{p,h} q)\, (\Pi^0_{p,h} v_{h,p})
     = \Q(\Pi^0_{p,h} v_{h,p}).  
  \end{equation}
The following result is an extension of $L^2$-error estimates for virtual elements, cf.  \cite[Section 3.4]{Brenner:2017}.

\begin{theorem}
\label{thm:qoi_apriori}
Let $z$ be the solution of the auxiliary 
problem 
\begin{equation}\label{pb:adjoint}
    -\nabla\cdot(\alpha \nabla z)=q\quad \textrm{in~}D, \qquad z=0 \quad\textrm{on~}\partial D.
\end{equation}
Assume that 
$f\in H^{s-1}(D)$, $q\in H^{t-1}(D)$, 
$u\in H^{s+1}(D)$, $z\in H^{t+1}(D)$, $s,t\ge 1$, 
and
\begin{gather}
    \label{eq:bounds_uz}
    \norm{z}{H^{t+1}(D)}\lesssim\alpha_{min}^{-1}\norm{q}{H^{t-1}(D)}, \quad
    \norm{u}{H^{s+1}(D)}\lesssim\alpha_{min}^{-1}\norm{f}{H^{s-1}(D)}.    
    \end{gather}
Then, setting~$\mu=\min\{s,t,p\}$, we have 
\begin{align}
    |\Q(u-u_{h,p})|&\lesssim 
\frac{s^\star(p)}{s_\star(p)}\left(\frac{\alpha_{max}}{\alpha_{min}}\right)^2 \alpha_{min}^{-1}
\frac{h^{2\mu}}{p^{2\,\min\{s,t-1\}}}
\left(\norm{q}{H^{t-1}(D)}^2+\norm{f}{H^{s-1}(D)}^2\right),
      \label{Qoi:error:1}\\
    |\Q(u)-\Q_h (u_{h,p})|&\lesssim
        \frac{s^\star(p)}{s_\star(p)}\left(\frac{\alpha_{max}}{\alpha_{min}}\right)^2 \alpha_{min}^{-1}
\frac{h^{2\mu}}{p^{2\,\min\{s,t-1\}}}
\left(\norm{q}{H^{t-1}(D)}^2+\norm{f}{H^{s-1}(D)}^2\right).
        \label{Qoi:error:2}
    \end{align}
\end{theorem}

\begin{proof}
    Let us set~$\mus=\min\{s,p\}$ and~$\mut=\min\{t,p\}$. We denote by $z_I$ the VE function in~$W_{h,p}$ corresponding to~$z\in H^{t+1}(D)$ as in~\eqref{eq:inter-estimate},
    for which there holds 
    $|z-z_I|_{H^1(D)}\lesssim \frac{h^{\mut}}{p^{t-1}}\norm{z}{H^{t+1}(D)}$.
    Using the auxiliary 
    problem \eqref{pb:adjoint}, we have
    \begin{eqnarray}
        |\Q(u-u_{h,p})|=|A(z,u-u_{h,p})|
        &\le& |A(z-z_I,u-u_{h,p})|+|A(z_I,u-u_{h,p})|\\&=:&(I)+(II).
    \end{eqnarray}
    By the continuity of $A$, \eqref{vem:estimate}, \eqref{eq:inter-estimate}, and~\eqref{eq:bounds_uz},
    there holds 
    \begin{gather*}
        \begin{aligned}
        (I)&\leq \alpha_{max}|u-u_{h,p}|_{H^1(D)}|z-z_I|_{H^1(D)}
        \lesssim \frac{s^\star(p)}{s_\star(p)}\frac{\alpha_{max}^2}{\alpha_{min}}
        \frac{h^{\mus+\mut}}{p^{s+t-1}} \norm{u}{H^{s+1}(D)} \norm{z}{H^
        {t+1}(D)}\\
        &\lesssim \frac{s^\star(p)}{s_\star(p)}\left(\frac{\alpha_{max}}{\alpha_{min}}\right)^2\alpha_{min}^{-1}
   \, \frac{h^{\mus+\mut}}{p^{s+t-1}} 
        \norm{f}{H^{s-1}(D)} \norm{q}{H^{t-1}(D)}. 
        \end{aligned}
    \end{gather*}
    For the second term, we write
$$
(II)\le \Big|\sum_E A^E(z_I-\Pi^{\nabla,E}_{p}z_I,u-u_{h,p})\Big|+\Big|\sum_E A^E(\Pi^{\nabla,E}_{p}z_I,u-u_{h,p})\Big|=: (III) + (IV),
$$
and estimate~$(III)$ and~$(IV)$ separately. Before doing so, we prove that
\begin{align}
\label{eq:new1}
\sum_E \seminorm{z_I-\Pi^{\nabla,E}_{p}z_I}{H^1(E)}^2 & \lesssim
\frac{h^{2\mut}}{p^{2 t-2}}\norm{z}{H^{t+1}(D)}^2,\\
\label{eq:new2}
\sum_E \seminorm{u_{h,p}-\Pi^{\nabla,E}_{p}u_{h,p}}{H^1(E)}^2 & \lesssim
\frac{h^{2\mus}}{p^{2s}}\norm{u}{H^{s+1}(D)}^2
\end{align}
To prove~\eqref{eq:new1}, we use the triangle inequality, the stability of~$\Pi^{\nabla,E}_{p}$ in~\eqref{eq:proj_stability_nabla}, and the estimates~\eqref{eq:inter-estimate} and~\eqref{eq:proj_approx_nabla}:
   \begin{gather*}
        \begin{aligned}
\sum_E \seminorm{z_I-\Pi^{\nabla,E}_{p}z_I}{H^1(E)}^2 & \le \sum_E \left(\seminorm{z_I-z}{H^1(E)}+\seminorm{z-\Pi^{\nabla,E}_{p}z}{H^1(E)}+\seminorm{\Pi^{\nabla,E}_{p}(z-z_I)}{H^1(E)}\right)^2\\
&\lesssim \sum_E \left(\seminorm{z_I-z}{H^1(E)}+ \seminorm{z-\Pi^{\nabla,E}_{p}z}{H^1(E)}\right)^2\\
&\lesssim \left(\seminorm{z_I-z}{H^1(D)}^2+ \seminorm{z-\Pi^{\nabla,E}_{p}z}{H^1(\T_h)}^2\right)\\
&\lesssim \frac{h^{2\mut}}{p^{2 t-2}}\norm{z}{H^{t+1}(D)}^2.
       \end{aligned}
    \end{gather*}  
The proof of~\eqref{eq:new2} follows the same lines, replacing~$z_I$ by~$u_{h,p}$, $z$ by~$u$, and using~\eqref{vem:estimate} instead of~\eqref{eq:inter-estimate}.    
We now estimate~$(III)$.
Using the continuity of~$A^E$, the discrete Cauchy-Schwarz inequality, and the estimates~\eqref{eq:new1} and~\eqref{vem:estimate}, we obtain
\begin{align*}
(III) & \le \alpha_{max}\sum_E  \seminorm{z_I-\Pi^{\nabla,E}_{p}z_I}{H^1(E)}\seminorm{u-u_{h,p}}{H^1(E)}\\
&\le \alpha_{max}\Big(\sum_E  \seminorm{z_I-\Pi^{\nabla,E}_{p}z_I}{H^1(E)}^2\Big)^{1/2} \seminorm{u-u_{h,p}}{H^1(D)}\\
&\lesssim \frac{s^\star(p)}{s_\star(p)}\frac{\alpha_{max}^2}{\alpha_{min}}
\frac{h^{\mus+\mut}}{p^{s+t-1}}\norm{z}{H^{t+1}(D)}
\norm{u}{H^{s+1}(D)}\\
&\lesssim \frac{s^\star(p)}{s_\star(p)}\left(\frac{\alpha_{max}}{\alpha_{min}}\right)^2 \alpha_{min}^{-1}
\frac{h^{\mus+\mut}}{p^{s+t-1}}\norm{q}{H^{t-1}(D)}
\norm{f}{H^{s-1}(D)},
\end{align*}
where, in the last step, we used~\eqref{eq:bounds_uz}.
In order to estimate~$(IV)$, let us preliminary observe that there holds
\begin{equation*} 
\begin{aligned}
         A^E(z_I,\Pi^{\nabla,E}_{p}u_{h,p})&=A^E(\Pi^{\nabla,E}_{p}z_I,\Pi^{\nabla,E}_{p} u_{h,p})
         \\
         &=A_h^E(z_I,u_{h,p}) - 
         S^E((I-\Pi^{\nabla,E}_{p})z_I,(I-\Pi^{\nabla,E}_{p})u_{h,p}),\\
    \end{aligned}
\end{equation*}
    where we employed the orthogonality property of $\Pi^{\nabla,E}_{p}$ and the definition of~$A^E_h$. 
Hence, we have
\begin{equation*} 
\begin{aligned}
&(IV)=\Big|\sum_E A^E(z_I,\Pi^{\nabla,E}_{p}(u-u_{h,p}))\Big|\\
&\qquad =\Big|\sum_E \left(A^E(z_I,\Pi^{\nabla,E}_{p}u)-A_h^E(z_I,u_{h,p}) + 
S^E((I-\Pi^{\nabla,E}_{p})z_I,(I-\Pi^{\nabla,E}_{p})u_{h,p})\right)\Big| 
\\
&\qquad =\Big|\sum_E \left(A^E(z_I,\Pi^{\nabla,E}_{p}u-u)+
S^E((I-\Pi^{\nabla,E}_{p})z_I,(I-\Pi^{\nabla,E}_{p})u_{h,p})\right)\\
& \qquad\qquad +F(z_I)-F_h(z_I)\Big|\\
&\qquad \le \Big|\sum_E A^E(z_I - \Pi^{\nabla,E}_{p} z_I ,\Pi^{\nabla,E}_{p}u - u)\Big|+ \Big|\sum_E 
S^E((I-\Pi^{\nabla,E}_{p})z_I,(I-\Pi^{\nabla,E}_{p})u_{h,p})\Big| \\
& \qquad\qquad+|F(z_I)-F_h(z_I)|.
    \end{aligned}
\end{equation*}
Since 
  \begin{gather*}
        \begin{aligned}
A^E(z_I - \Pi^{\nabla,E}_{p} z_I ,\Pi^{\nabla,E}_{p}u - u)
&\le \alpha_{max}\seminorm{z_I - \Pi^{\nabla,E}_{p} z_I}{H^1(E)}\seminorm{\Pi^{\nabla,E}_{p}u - u}{H^1(E)}\\
&\lesssim \alpha_{max} \seminorm{z_I - \Pi^{\nabla,E}_{p} z_I}{H^1(E)}^2
+\alpha_{max} \seminorm{\Pi^{\nabla,E}_{p}u - u}{H^1(E)}^2,
    \end{aligned}
    \end{gather*} 
using~\eqref{eq:new1} and the estimate in~\eqref{eq:proj_approx_nabla}, we obtain 
  \begin{gather*}
        \begin{aligned}
\Big|\sum_E A^E(z_I &- \Pi^{\nabla,E}_{p} z_I ,\Pi^{\nabla,E}_{p}u - u)\Big|\\
\lesssim &\alpha_{max}\frac{h^{2\mut}}{p^{2t-2}}\norm{z}{H^{t+1}(D)}^2+\alpha_{max} \frac{h^{2\mus}}{p^{2 s}}\norm{u}{H^{s+1}(D)}^2\\
\lesssim &\frac{\alpha_{max}}{\alpha_{min}}\alpha_{min}^{-1}\frac{h^{2\,\min\{\mus,\mut\}}}{p^{2\,\min\{s,t-1\}}}
\left(\norm{q}{H^{t-1}(D)}^2+\norm{f}{H^{s-1}(D)}^2\right),
    \end{aligned}
    \end{gather*}  
    where, in the last step, we used again~\eqref{eq:bounds_uz}.
 Furthermore, the bilinearity and symmetry of~$S^E$, together with the equivalence~\eqref{eq:equiv}, imply
  \begin{gather*}
        \begin{aligned}
&S^E((I-\Pi^{\nabla,E}_{p})z_I,(I-\Pi^{\nabla,E}_{p})u_{h,p})\\ 
&=\frac{1}{2}S^E\left((I-\Pi^{\nabla,E}_{p})z_I+(I-\Pi^{\nabla,E}_{p})u_{h,p},(I-\Pi^{\nabla,E}_{p})z_I+(I-\Pi^{\nabla,E}_{p})u_{h,p}\right)\\
&\quad -\frac{1}{2}S^E((I-\Pi^{\nabla,E}_{p})z_I,(I-\Pi^{\nabla,E}_{p})z_I)
 -\frac{1}{2}S^E((I-\Pi^{\nabla,E}_{p})u_{h,p},(I-\Pi^{\nabla,E}_{p})u_{h,p})\\
&\lesssim s^\star(p)A^E\left((I-\Pi^{\nabla,E}_{p})z_I+(I-\Pi^{\nabla,E}_{p})u_{h,p},(I-\Pi^{\nabla,E}_{p})z_I+(I-\Pi^{\nabla,E}_{p})u_{h,p}\right)\\
 &\le s^\star(p)\alpha_{max}\seminorm{(I-\Pi^{\nabla,E}_{p})z_I+(I-\Pi^{\nabla,E}_{p})u_{h,p}}{H^1(E)}^2\\
 &\lesssim  s^\star(p)\alpha_{max} \seminorm{(I-\Pi^{\nabla,E}_{p})z_I}{H^1(E)}^2
 +s^\star(p)\alpha_{max} \seminorm{(I-\Pi^{\nabla,E}_{p})u_{h,p}}{H^1(E)}^2.
      \end{aligned}
    \end{gather*} 
Therefore, using~\eqref{eq:new1} and~\eqref{eq:new2}, we obtain
 \begin{gather*}
        \begin{aligned}
\Big|\sum_E  & S^E((I-\Pi^{\nabla,E}_{p})z_I,(I-\Pi^{\nabla,E}_{p})u_{h,p})\Big|\\
\lesssim \, &s^\star(p)\alpha_{max}\frac{h^{2\,\min\{\mus,\mut\}}}{p^{2\,\min\{s,t-1\}}}\left(\norm{z}{H^{t+1}(D)}^2+\norm{u}{H^{s+1}(D)}^2\right)\\
\lesssim \, &s^\star(p)\frac{\alpha_{max}}{\alpha_{min}}\alpha_{min}^{-1}\frac{h^{2\,\min\{\mus,\mut\}}}{p^{2\,\min\{s,t-1\}}}\left(\norm{q}{H^{t-1}(D)}^2+\norm{f}{H^{s-1}(D)}^2\right),
      \end{aligned}
    \end{gather*}     
where, in the last step, we used again~\eqref{eq:bounds_uz}.
Finally, for $|F(z_I)-F_h(z_I)|$, we have  
    \begin{eqnarray}
    |F(z_I)-F_h(z_I)|&=&|(F-F_h)(z_I-z)+(F-F_h)(z)|  \nonumber\\
    &=&\Big|\sum_{E}\int_E (f-\Pi^{0,E}_{p}f)(z_I-z) + \sum_{E} \int_E (f-\Pi^{0,E}_{p}f)(z-\Pi^{0,E}_{p}z)\Big|\nonumber \\
&=&\Big|\sum_{E}\int_E (f-\Pi^{0,E}_{p}f)(I-\Pi^{0,E}_{p})(z_I-z) + \sum_{E} 
\int_E(f-\Pi^{0,E}_{p}f)(z-\Pi^{0,E}_{p}z)\Big|\nonumber\\
    &\le&
    \|f-\Pi^{0}_{h,p}f\|_{L^2(D)}
   ( 
   \| (I-\Pi^{0}_{h,p})(z_I-z)\|_{L^2(D)}
    +
    \|z-\Pi^{0}_{h,p}z\|_{L^2(D)})
    \nonumber\\
    &\lesssim&
    {\|f-\Pi^{0}_{h,p}f\|_{L^2(D)}\left(\frac{h}{{p}}
    \seminorm{z-z_I}{H^1(D)}
    +\|z-\Pi^{0}_{p}z\|_{L^2(D)}\right)}\nonumber\\
    &\lesssim& {\frac{h^{\min\{s,p+2\}+\mut}}{p^{s+t-1}}\norm{f}{H^{s-1}(D)}\norm{z}{H^{t+1}(D)}}\\
    &\lesssim& \alpha_{min}^{-1}\frac{h^{\min\{s,p+2\}+\mut}}{p^{s+t-1}}\norm{f}{H^{s-1}(D)}\norm{q}{H^{t-1}(D)},
    \nonumber
\end{eqnarray}
where we also employed~\eqref{eq:L2H1} and the error estimate $\|f-\Pi^{0,E}_{p}f\|_{L^2(D)}\lesssim \frac{h^{\min\{s-1,p+1\}}}{p^{s-1}}\norm{f}{H^{s-1}(D)}$, which holds as $f\in H^{s-1}(D)$.
Collecting the above results, we obtain, for the term~$(IV)$,
\begin{gather*}
\begin{aligned}
(IV)&\lesssim s^\star(p)\frac{\alpha_{max}}{\alpha_{min}}\alpha_{min}^{-1}\frac{h^{2\,\min\{\mus,\mut\}}}{p^{2\,\min\{s,t-1\}}}\left(\norm{q}{H^{t-1}(D)}^2+\norm{f}{H^{s-1}(D)}^2\right).
\end{aligned}
\end{gather*}
This, together with the estimate obtained for~$(III)$, gives
\[
(II)\lesssim \frac{s^\star(p)}{s_\star(p)}\left(\frac{\alpha_{max}}{\alpha_{min}}\right)^2 \alpha_{min}^{-1}
\frac{h^{2\,\min\{\mus,\mut\}}}{p^{2\,\min\{s,t-1\}}}
\left(\norm{q}{H^{t-1}(D)}^2+\norm{f}{H^{s-1}(D)}^2\right),
\]
which, combined with the estimate obtained for~$(I)$, yields
\[
|\Q(u-u_{h,p})|\lesssim 
\frac{s^\star(p)}{s_\star(p)}\left(\frac{\alpha_{max}}{\alpha_{min}}\right)^2 \alpha_{min}^{-1}
\frac{h^{2\,\min\{\mus,\mut\}}}{p^{2\,\min\{s,t-1\}}}
\left(\norm{q}{H^{t-1}(D)}^2+\norm{f}{H^{s-1}(D)}^2\right).
\]
Then, estimate~\eqref{Qoi:error:1} follows, owing to~$\min\{\mus,\mut\}=\min\{s,t,p\}$.

\medskip

To obtain \eqref{Qoi:error:2},
    we observe that there holds
\begin{eqnarray*}
        |\Q(u)-\Q_h(u_{h,p})| &=& |\Q(u)-\Q_h(u) + \Q_h(u)-\Q_h(u_{h,p})|\nonumber\\
        &=& \Big|\int_D (q - \Pi^0_{p,h} q) u + \int_D \Pi^0_{p,h} q(u-u_{h,p})\Big| \nonumber\\
        &=& \Big|\int_D (q - \Pi^0_{p,h} q)(u- \Pi^0_{p,h}u)+\int_D (\Pi^0_{p,h}q-q)(u-u_{h,p})+\Q(u-u_{h,p})\Big| \nonumber\\
        &\le& \|q-\Pi^0_{p,h} q\|_{L^2(D)}\left(\|u-\Pi^0_{p,h} u\|_{L^2(D)}+\|u-u_{h,p}\|_{L^2(D)}\right)+\Q(u-u_{h,p}).
    \end{eqnarray*}
Then, using the $L^2$ projection error estimates (notice that  $q\in H^{t-1}(D)$ implies 
$$\|q-\Pi^{0,E}_{p}q\|_{L^2(D)}\lesssim \frac{h^{\min\{t-1,p+1\}}}{{p^{t-1}}} \norm{q}{H^{t-1}(D)}$$ and the estimate in~\eqref{vem:estimate}, we have
\begin{gather*}
    \begin{aligned}
        |\Q(u)-\Q_h(u_h)|
        & \lesssim \frac{\alpha^*(p)}{s_\star(p)}\frac{\alpha_{max}}{\alpha_{min}}\frac{h^{\mus+\min\{t,p+2\}}}{p^{s+t}} \norm{q}{H^{t-1}(D)}\norm{u}{H^{s+1}(D)} + \Q(u-u_{h,p})\\
        & \lesssim \frac{\alpha^*(p)}{s_\star(p)}\frac{\alpha_{max}}{\alpha_{min}}\alpha_{min}^{-1}\frac{h^{\mus+\min\{t,p+2\}}}{p^{s+t}} \norm{q}{H^{t-1}(D)} \norm{f}{H^{s-1}(D)} + \Q(u-u_{h,p}).
    \end{aligned}
\end{gather*}
This, together with~\eqref{Qoi:error:1} gives~\eqref{Qoi:error:2}.
   \end{proof} 

\begin{remark}\label{rem:classical}
The presence of~$p^{2\min\{s,t-1\}}$ in the denominators in estimates~\eqref{Qoi:error:1} and~\eqref{Qoi:error:2} of Theorem~\ref{thm:qoi_apriori} is a consequence of the use, in the proof, of the best approximant~$z_I$ in the ``enhanced'' VE space~$W_{h,p}$,
which was proven in to satisfy~\eqref{eq:inter-estimate}. 
However, a closer inspection of the proof of~Theorem~\ref{thm:qoi_apriori} shows that only the degrees of freedom of~$z_I$ are actually used, together with the consistency of method~\eqref{pb: vem} when tested with the function~$z_I$. Therefore, if the right-hand side~$F_h(v_{h,p})$ of method~\eqref{pb: vem} were defined as
\[
F_h(v_{h,p})
=\sum_{E\in \T_h}\int_E \Big(\Pi^{0,E}_{p-2} f\Big)\, v_{h,p},
\]
that is employing~$\Pi^{0,E}_{p-2}$ instead of~$\Pi^{0,E}_{p}$, one could use the best approximant in the ``classical'' VE space (see~\cite{BCMR:2016} for its definition), say~$z_I^c$, instead of~$z_I$, without modifying the argument. In fact, since only the degrees of freedom of the approximant enter the analysis, there is no difference between choosing~$z_I$ or~$z_I^c$. 
We stress that modifying the definition of the right-hand side~$F_h$ is necessary for this argument. Indeed, with the original definition of~$F_h$, the required consistency property would not longer hold when testing with~$z_I^c$. Consequently, the estimate of term~$(IV)$ in the proof of Theorem~\ref{thm:qoi_apriori} could not be reproduced after replacing~$z_I$ by~$z_I^c$.  
The advantage of this alternative approach is that, for~$z_I^c$, the analogue of estimate~\eqref{eq:inter-estimate} would involve~$p^{\min\{s,t\}}$ in the denominator, instead of~$p^{\min\{s,t-1\}}$; see~\cite[Lemma 4.3]{BCMR:2016},
resulting in the presence of~$p^{2\min\{s,t\}}$ instead of~$p^{2\min\{s,t-1\}}$ in the denominators in estimates~\eqref{Qoi:error:1} and~\eqref{Qoi:error:2}.
\eremk
\end{remark}

\section{The MC-VE method}
\label{sec:mcvem}

In this section, we propose 
a method for approximating the expected value of the solution~$u$ of the elliptic problem with stochastic coefficient~\eqref{eq:strong_pde}, as well as the expected value of $\Q(u)$.\\

A widely used strategy to discretize integrals in probability is the Monte Carlo method, which relies on the following operator: 
\begin{equation}
    \mc{M}[w]=\frac{1}{M}\sum_{i=1}^M w^{(i)},
\end{equation}
where $w^{(i)}= w(\omega_i)$, $i=1,\ldots, M$, are $M$ independent identically distributed (i.i.d.) samples of the stochastic function $w(\omega)$. {Possible examples are $w^{(i)}=u(\omega_i,\cdot)$ and $w^{(i)}=\Q(u(\omega_i,\cdot))$, where $u(\omega_i,\cdot)$ is the unique solution of \eqref{eq:weak_pde} with $\omega=\omega_i$.}
We recall from \cite{BarthSchwabZollinger2011} that the following holds (note that $\E[w]$ is deterministic, whereas $\mc{M}[w]$ is a random variable because its value depends on the considered samples):
\begin{equation}
    \label{eq:mc_error_u}
    \|\E[w]-\mc{M}[w]\|_{L^2(\Omega,V)} 
    \leq M^{-1/2}
    \|w\|_{L^2(\Omega,V)} 
    \quad\forall w\in L^2(\Omega,V). 
\end{equation}
As the aim of the present work is to approximate the expected value of the
solution~$u$ to problem~\eqref{eq:strong_pde} as well as its quantity of interest $\Q(u)$, we combine MC with the VE method of order~$p$ defined over a polytopal mesh of meshsize~$h$. As a result, we get the following MC-VE estimators for $\E[u]$ and $\E[\Q(u)]$: 
\begin{align}
    \label{eq:mcvem_estimator_u}
    \mc{M}[\Pi^\nabla_{h,p} u_{h,p}]&= \frac{1}{M}\sum_{i=1}^M \Pi_{h,p}^\nabla u_{h,p}^{(i)} \in L^2(\Omega,\VbTh),\\
    \label{eq:mcvem_estimator_q}
    \mc{M}[\mathcal{Q}(\Pi^0_{h,p}u_{h,p})]&{= \frac{1}{M}\sum_{i=1}^M \mathcal{Q}(\Pi_{h,p}^0 u_{h,p}^{(i)}) \in L^2(\Omega)},
\end{align}
where $u_{h,p}^{(i)}= u_{h,p}(\omega_i,\cdot)$, $i=1,\ldots, M$, are $M$ i.i.d.~samples of the stochastic VE solution $u_{h,p}(\omega,\cdot)$, obtained by solving $M$ VE problems of the form \eqref{pb: vem}, each corresponding to a~sample~$a(\omega_i,\cdot)$ 
of the stochastic diffusion coefficient. 
We define the Bochner seminorm
\[
\norm{\cdot}{L^2(\Omega,\VbTh)}= 
\left(\int_\Omega | \cdot|_{H^1(\T_h)}^2 \right)^{1/2},
\]
which is a norm for functions with zero mean on each element of~$\T_h$. In~$L^2(\Omega,V)$, $\norm{\cdot}{L^2(\Omega,\VbTh)}=\norm{\cdot}{L^2(\Omega,V)}$.
Since in all the cases when it is used below it is actually a norm, we have defined it with the $\norm{\cdot}{}$ symbol.

\begin{remark}
Although the natural choice in~\eqref{eq:mcvem_estimator_u} would be $E_M[u_{h,p}]$ (and $E_M[\Q(u_{h,p}]$ in \eqref{eq:mcvem_estimator_q}), this is not practical, since the virtual function $u_{h,p}$ is not known in closed form. For this reason, we use the projection $\Pi^\nabla_{h,p} u_{h,p}$, which is computable from the DOFs.
\eremk
\end{remark}

In the following proposition, we derive estimates for the statistical approximation error for both the solution and the quantity of interest. For the latter, we use the analogue of the estimate derived in Theorem~\ref{thm:qoi_apriori} for a deterministic diffusion coefficient. In the case of problem~\eqref{eq:weak_pde}, this estimate becomes, with~$\mu=\min\{s,t,p\}$,
\begin{equation}
\label{eq:BQ}
\begin{split}
|\Q(u)-&\Q_h (u_{h,p})|    \\
&\lesssim \todo{\frac{s^\star(p)}{s_\star(p)}\left(\frac{a_{max}(\omega)}{a_{min}(\omega)}\right)^2 a_{min}(\omega)^{-1}}
\frac{h^{2\,\mu}}{p^{\todo{2\,\min\{s,t-1\}}}}
\left(\norm{q}{H^{t-1}(D)}^2+\norm{f}{H^{s-1}(D)}^2\right).
\end{split}
\end{equation} 

\begin{proposition}
    \label{prop:mcvem}
    Assume that $u\in L^2(\Omega,H^{s+1}(D))$, $s\ge 1$.
    Then, the following error bound on $\E[u]$ holds, with~$\mus=\min\{s,p\}$:
    \begin{gather}
    \label{eq:mcvem_error_u_2}
    \begin{aligned}
        &\|\E[u]-\mc{M}[\Pi^\nabla_{h,p}u_{h,p}]\|_{L^2(\Omega,\VbTh)} 
        \\&\lesssim 
        \todo{\frac{s^\star(p)}{s_\star(p)}}\max\left\{\norm{\frac{a_{max}}{a_{min}}}{L^2(\Omega)},
        \norm{a_{min}^{-1}}{L^2(\Omega)}\right\}
    \Bigg(\frac{h^{\todo{\mus}}}{p^s} + M^{-1/2}\Bigg)
        \left(\norm{u}{L^2(\Omega,H^{s+1}(D))}+\norm{f}{L^2(D)}\right),
    \end{aligned}
    \end{gather}
    where~$\norm{\cdot}{L^2(\Omega,H^{s+1}(D))}$ is the standard Bochner norm.
    Moreover, under the assumptions of Theorem \ref{thm:qoi_apriori}, the following error bound on $\E[\Q(u)]$ holds, with~$\mu=\min\{s,t,p\}$: 
    \begin{equation}
    \label{eq:mcvem_error_q}
    \begin{split}
        &\|\E[\mathcal Q(u)]-\mc{M}[\mathcal{Q}(\Pi^0_{h,p}u_{h,p})]\|_{L^2(\Omega)} \\
        &\lesssim \todo{\frac{s^\star(p)}{s_\star(p)}}
\norm{\Big(\frac{a_{max}}{a_{min}}\Big)^{\todo{2}}}{L^2(\Omega)}\norm{a_{min}^{-1}}{L^2(\Omega)}\Bigg(\frac{h^{2\todo{\mu}}}{p^{\todo{2\,\min\{s,t-1\}}}} + M^{-1/2}\Bigg)\left(\norm{q}{H^{t-1}(D)}^2+\norm{f}{H^{\todo{s}-1}(D)}^2\right).
        \end{split}
    \end{equation}
\end{proposition}

\begin{proof}
The proof of \eqref{eq:mcvem_error_u_2} follows the same steps as the proof 
of~\cite[Theorem~4.3]{BarthSchwabZollinger2011}. 
We include the proof here for completeness. By the triangle inequality, we get
\begin{gather*}
\begin{aligned}
    &\|\E[u]-\mc{M}[\Pi^\nabla_{h,p}u_{h,p}]\|_{L^2(\Omega,\VbTh)}\\
    &\leq 
    \underbrace{\|\E[u]-\E[\Pi^\nabla_{h,p}u_{h,p}]\|_{L^2(\Omega,\VbTh)}}_{(I)}
    + \underbrace{\|\E[\Pi^\nabla_{h,p}u_{h,p}]-\mc{M}[\Pi^\nabla_{h,p}u_{h,p}]\|_{L^2(\Omega,\VbTh)}}_{(II)}.
\end{aligned}
\end{gather*}
We estimate $(I)$ by using the approximation and stability properties \eqref{eq:proj_approx_nabla} and~\eqref{eq:proj_stability_nabla} of $\Pi^\nabla_{h,p}$, as well as the error estimate \eqref{vem:estimate} of the VE method: 
\begin{gather*}
\begin{aligned}
    \|\E[u]-\E[\Pi^\nabla_{h,p}u_{h,p}]\|_{L^2(\Omega,\VbTh)}
    &= \|\E[u-\Pi^\nabla_{h,p}u_{h,p}]\|_{L^2(\Omega,\VbTh)}\\
    &\leq \E[\seminorm{u-\Pi^\nabla_{h,p}u_{h,p}}{\VbTh}]\\
    &\leq \E[\seminorm{u-\Pi^\nabla_{h,p}u}{\VbTh}] +\E[\seminorm{\Pi^\nabla_{h,p}(u-u_{h,p})}{\VbTh}]\\
    &\lesssim \todo{\frac{s^\star(p)}{s_\star(p)}}\norm{\frac{a_{max}}{a_{min}}}{L^2(\Omega)}\frac{h^{\todo{\mus}}}{{p^s}}\norm{u}{L^2(\Omega,H^{s+1}(D))}.
\end{aligned}
\end{gather*}
For $(II)$, we use the MC error estimate~\eqref{eq:mc_error_u} and the stability estimates~\eqref{eq:proj_stability_nabla} and~\eqref{eq:H1} obtaining
\begin{gather*}
    \begin{aligned}
(II)&\le M^{-1/2}\norm{\Pi^\nabla_{h,p}u_{h,p}}{L^2(\Omega,\VbTh)}
\lesssim M^{-1/2}\norm{u_{h,p}}{L^2(\Omega,\VbTh)}\\
&\lesssim \todo{s_\star(p)^{-1}}\norm{ a_{min}^{-1}}{L^2(\Omega)} M^{-1/2}\norm{f}{L^2(D)},
      \end{aligned}
\end{gather*}
from which~\eqref{eq:mcvem_error_u_2} follows, \todo{taking into account that~$s_\star(p)^{-1}\le \frac{s^\star(p)}{s_\star(p)}$.}

The error bound \eqref{eq:mcvem_error_q} follows similarly, by adding and subtracting the term $\E[\mathcal Q(\Pi^0_{h,p} u_{h,p})]$ into $\|\E[\mathcal Q(u)]-\mc{M}[\mathcal Q(\Pi^0_{h,p} u)]\|_{L^2(\Omega)}$:
\begin{gather*}
\begin{aligned}
    &\|\E[\Q(u)]-\mc{M}[\Q(\Pi^0_{h,p}u_{h,p})]\|_{L^2(\Omega)}\\
    &\leq 
    \underbrace{\|\E[\Q(u)]-\E[\Q(\Pi^0_{h,p}u_{h,p})]\|_{L^2(\Omega)}}_{(I)}
    + \underbrace{\|\E[\Q(\Pi^0_{h,p}u_{h,p})]-\mc{M}[\Q(\Pi^0_{h,p}u_{h,p})]\|_{L^2(\Omega)}}_{(II)}.
\end{aligned}
\end{gather*}
Thanks to the linearity of the operator $\E[\cdot]$, and using $\Q(\Pi^0_{h,p}u_{h,p})=\Q_h(u_{h,p})$ (see \eqref{eq:Qh_identity}), there holds:
\begin{gather*}
    \begin{aligned}
        (I)&
        =\| \E[\Q(u)-\Q_h(u_{h,p})] \|_{L^2(\Omega)}= \E[\todo{|}\Q(u)-\Q_h(u_{h,p})\todo{|}]\\
        &\lesssim \todo{\frac{s^\star(p)}{s_\star(p)}}\E\Big[\Big(\frac{a_{max}}{a_{min}}\Big)^{\todo{2}} a_{min}^{-1}\Big]
        \frac{h^{2\todo{\mu}}}{p^{\todo{2\,\todo{\min\{s,t-1\}}}}}  \left(\norm{q}{H^{\todo{t}-1}(D)}^2+\norm{f}{H^{\todo{s}-1}(D)}^2\right), 
    \end{aligned}
\end{gather*}
where, in the last step, we applied \eqref{eq:BQ}.
From~\eqref{eq:mc_error_u}, the definition of~$\Q$, the Cauchy-Schwarz inequality, and the $L^2(D)$ stability of~$\Pi^0_{h,p}$, we have
\begin{gather*}
    \begin{aligned}
(II)\lesssim M^{-1/2}\norm{\Q(\Pi^0_{h,p}u_{h,p})}{L^2(\Omega)}
\leq M^{-1/2} \norm{q}{L^2(D)} \norm{u_{h,p}}{L^2(\Omega,L^2(D))}.
    \end{aligned}
\end{gather*}
Using the Poincar\'e inequality and the stability property in~\eqref{eq:H1}, we obtain
\begin{gather*}
    \begin{aligned}
(II)\lesssim \todo{s_\star(p)^{-1}}\norm{a_{min}^{-1}}{L^2(\Omega)} M^{-1/2} \norm{q}{L^2(D)}\norm{f}{L^2(D)}.
   \end{aligned}
\end{gather*}
The above estimates of $(I)$ and~$(II)$ give \eqref{eq:mcvem_error_q}, taking onto account that 
\begin{gather*}
    \begin{aligned}
\max\left\{\todo{\frac{s^\star(p)}{s_\star(p)}}\E\Big[\Big(\frac{a_{max}}{a_{min}}\Big)^{\todo{2}} a_{min}^{-1}\Big], \todo{s_\star(p)^{-1}}\norm{a_{min}^{-1}}{L^2(\Omega)}\right\}
\\
\le \todo{\frac{s^\star(p)}{s_\star(p)}}
\norm{\Big(\frac{a_{max}}{a_{min}}\Big)^{\todo{2}}}{L^2(\Omega)}\norm{a_{min}^{-1}}{L^2(\Omega)}.
   \end{aligned}
\end{gather*}
\end{proof}

\begin{remark}
The regularity assumption in Proposition~\ref{prop:mcvem} is satisfied provided that the diffusion coefficient~$a$ is sufficiently smooth. For instance, if~$a(\omega,\cdot)\in W^{1,\infty}(D)$ satisfies 
$\|a(\omega,\cdot)\|_{W^{1,\infty}(D)}\leq C$ uniformly over $\omega\in\Omega$,
the elliptic regularity theory ensures that $\|u(\omega,\cdot)\|_{H^2(D)}$ is bounded uniformly over $\omega\in\Omega$, and thus fulfills $u\in L^2(\Omega,H^2(D))$. This holds, for example, if~$a(\omega,x)$ is given by the following Karhunen--Lo\`eve-type expansion:
\begin{gather*}
    a(\omega,x)=\bar a(x) + \sum_{j=1}^J \beta_j \varphi_j(x) Y_j(\omega),
\end{gather*}
where $\bar a\in W^{1,\infty}(D)$ is uniformly positive, $\{Y_j\}_{j=1}^J$ is a set of independent uniformly distributed random variables $Y_j\sim\mathcal{U}([-1,1])$, $\{\varphi_j\}_{j=1}^J\subset W^{1,\infty}(D)$, 
and~$\{\beta_j\}_{j=1}^J$ is a set of non-negative coefficients such that $\beta_j\|\varphi_j\|_{W^{1,\infty}(D)}\leq C\, j^{-n}$ for some $n>1$ and constant $C>0$ independent of $j$. 

For diffusion coefficients with $\omega$-dependent lower and upper bounds $a_{min}(\omega)$, $a_{max}(\omega)$, proving that $u\in L^2(\Omega,H^2(D))$ requires a more delicate analysis, since $\|u(\omega,\cdot)\|_{H^2(D)}$ has an $\omega$-dependent upper bound. In~\cite[Theorem 2.1]{TeckentrupScheichlGilesUllmann}, it is proved that, if $a\in L^p(\Omega,C^t(\bar D))$ and $f\in H^{t-1}(D)$ for some $0<t\leq 1$ and any $p<\infty$, then $u\in L^p(\Omega,H^{2}(D))$. For instance, this applies when the diffusion coefficient in problem~\eqref{eq:strong_pde} is lognormal, namely, $a(\omega,x)=\exp(Y(\omega,x))$, with~$Y$ being a Gaussian random field whose mean is H\"older continuous and whose covariance function is Lipschitz continuous.
\eremk
\end{remark}

\begin{remark}
\label{rem:mc_Mopt}
The choice~$M={\mathcal O}(p^{2s}\,h^{-2\todo{\mus}})$ balances the MC error (${\mathcal O}(M^{-1/2})$) with the VE discretization error~(${\mathcal O}(h^{\todo{\mus}}\,p^{-s})$) in~\eqref{eq:mcvem_error_u_2}, resulting in the overall error of~${\mathcal O}(h^{\todo{\mus}}\,p^{-s}\todo{\frac{s^\star(p)}{s_\star(p)}})$. This is the optimal choice of~$M$, given \todo{the smoothness of~$u$,} the mesh size~$h$ and the 
\todo{order}~$p$, to achieve a prescribed tolerance. {Analogously, the choice $M={\mathcal O}(p^{\todo{4\,\todo{\min\{s,t-1\}}}}\,h^{-4\todo{\mu}})$ balances the MC error (${\mathcal O}(M^{-1/2})$) with the VE discretization error~(${\mathcal O}(h^{2\todo{\mu}}\,p^{-{\todo{2\,\todo{\min\{s,t-1\}}}}}\todo{\frac{s^\star(p)}{s_\star(p)}})$) in~\eqref{eq:mcvem_error_q}.}
\eremk
\end{remark}

{\begin{remark}
One could alternatively define a
MC-VE estimator for $\E[u]$ employing the $L^2$ projector: 
\begin{align}
\label{EM:new_def}
    \mc{M}[\Pi^0_{h,p} u_{h,p}]&= \frac{1}{M}\sum_{i=1}^M \Pi_{h,p}^0 u_{h,p}^{(i)} \in L^2(\Omega,L^2(D)),
\end{align}
for which the following bound can be proved along the same steps as in the proof of Proposition~\ref{prop:mcvem}:
\begin{gather}
\label{eq:bound_Eu_Pi0}
\|\E[u]-\mc{M}[\Pi^0_{h,p}u_{h,p}]\|_{L^2(\Omega,L^2(D))}\lesssim\Bigg(\frac{h^{\todo{\mus}+1}}{{p^{s+1}}} + M^{-1/2}\Bigg).
\end{gather}
This estimator provides no control in the~$L^2(\Omega,H^1(\T_h))$ norm.
The use of definition~\eqref{EM:new_def} is particularly helpful in applications where 
multiple functionals $\Q$ of the solution of the form~\eqref{eq:q}
are of interest. Indeed, noting that $\mc{M}[\Q(\Pi^0_{h,p}u_{h,p})]=\Q(\mc{M}[\Pi^0_{h,p}u_{h,p}])$, one can compute the MC-VE estimator~$\mc{M}[\Pi^0_{h,p}u_{h,p}]$ for $\E[u]$, and then the MC-VE estimator for $\E[\Q(u)]$ can be obtained by simply applying the functional $\Q$ to~$\mc{M}[\Pi^0_{h,p}u_{h,p}]$, for all occurrences of~$\Q$, allowing for substantial computational savings. 

For the error analysis, we observe that the following holds 
\[
\begin{aligned}
\|\E[\mathcal Q(u)]-\mc{M}[\mathcal{Q}(\Pi^0_{h,p}u_{h,p})]\|_{L^2(\Omega)}
&= \|\Q\big(\E[u]-\mc{M}[\Pi^0_{h,p}u_{h,p}]\big)\|_{L^2(\Omega)}\\
&\leq \|q\|_{L^2(D)}\|\E[u]-\mc{M}[\Pi^0_{h,p}u_{h,p}]\|_{L^2(\Omega,L^2(D))},   
\end{aligned}
\]
where the equality follows from Fubini's theorem and the linearity of $\Q$, and the inequality follows from the Cauchy-Schwarz inequality. 
Hence, by using~\eqref{eq:bound_Eu_Pi0}, we could directly derive an upper bound of the error~$\|\E[\mathcal Q(u)]-\mc{M}[\mathcal{Q}(\Pi^0_{h,p}u_{h,p})]\|_{L^2(\Omega)}$. 
This bound, however, would not be sharp; cf.~\eqref{eq:mcvem_error_q}.
\eremk
\end{remark}}

To alleviate the computational burden of the MC-VE method without compromising accuracy, we study the multi-level $h$-version of the MC-VE method (MLMC-VE) in the following section. This approach considers multiple spatial resolution levels: the estimator is first computed on a coarse mesh 
using the MC-VE method,  
and subsequent correction terms are added to improve accuracy.

\section{The MLMC-VE method}
\label{sec:MLMC-VE}
The MLMC-VE method we introduce in this section relies on a 
sequence of nested polygonal meshes $\{\T_\ell\}_{\ell\ge 0}$ discretizing the domain $D$: 
\begin{gather}
\label{eq:nested_meshes}
\T_0 \subseteq\T_1\subseteq\cdots \subseteq \T_\ell\subseteq\cdots\subseteq\T_L\subseteq\cdots,
\end{gather}
where each $\T_\ell$ is a refinement of $\T_{\ell-1}$, and $\T_\ell$ has meshsize $h_\ell$. Let~$V_\ell$ denote the VE space of degree $p$ (fixed and independent of~$\ell$) on the mesh $\T_\ell$. The sequence of meshes~\eqref{eq:nested_meshes} induces the sequence $\{V_\ell\}_{\ell\ge 0}$ of conforming VE spaces, which are not nested, although their polynomial subspaces are. This structure guarantees improved approximation properties as~$\ell\rightarrow\infty$, almost surely in $\Omega$:
\begin{align*}
    \| u(\omega,\cdot)-{u_{h_{\ell+1},p}}(\omega,\cdot)\|_V
    &\leq \| u(\omega,\cdot)-{u_{h_{\ell},p}}(\omega,\cdot)\|_V,\\
    \| u(\omega,\cdot)-{u_{h_{\ell},p}}(\omega,\cdot)\|_V&
    \xrightarrow[]{\ell\rightarrow\infty}0,
\end{align*}
where $u(\omega,\cdot)$ denotes the solution of~\eqref{eq:weak_pde}, and ${u_{h_{\ell+1},p}}(\omega,\cdot)\in V_\ell$ denotes its VE approximation, i.e., the solution of \eqref{pb: vem}.
Since $p$ is fixed, to simplify the notation we set 
$\Pi^\nabla_\ell=
\Pi^\nabla_{h_\ell,p}$ {and $u_{h_\ell,p}= u_\ell$}.
Given the finite sequence $\{\T_\ell\}_{\ell=0}^L$ of the first $L+1$ nested meshes from~\eqref{eq:nested_meshes}, and the corresponding sequence $\{V_\ell\}_{\ell=0}^L$ of VE spaces, we define the MLMC-VE estimator for~$\E[u]$ 
as
\begin{align}
    \label{eq:mlmc_vem_u}
    \mlmc[u]&= \sum_{\ell=1}^L \mc{M_\ell}[w_\ell] 
    = \sum_{\ell=1}^L \frac{1}{M_\ell} \sum_{i=1}^{M_\ell }w^{(i)}_\ell \in L^2(\Omega,\VbTL), 
\end{align}
where
$$
w_\ell=\Pi^\nabla_\ell u_\ell
    -\Pi^\nabla_{\ell-1} u_{\ell-1}, \qquad
w^{(i)}_\ell=\Pi^\nabla_\ell u^{(i)}_\ell
    -\Pi^\nabla_{\ell-1} u^{(i)}_{\ell-1},$$
and $u_{0}^{(i)}= 0$. Note that the MLMC-VE estimator~\eqref{eq:mlmc_vem_u} consists of the sum over the~$L$ levels of the MC-VE estimators of the increments $w_\ell$, using a level-dependent number $M_\ell$ of samples. 

\begin{remark}
\label{rem:mlmc:1}
Since $u_\ell^{(i)}$ and $u_{\ell-1}^{(i)}$ are VE functions belonging to two different, non-nested VE spaces $V_\ell$ and $V_{\ell-1}$, respectively, an MLMC-VE estimator of the type
\[
\mlmc[u] = \sum_{\ell=1}^L 
\mc{M_\ell}
[u_\ell-u_{\ell-1}],
\]
would not be computable. This is why, in our definition~\eqref{eq:mlmc_vem_u}, we insert the projection operators $\Pi^\nabla_\ell$ and $\Pi^\nabla_{\ell-1}$, which makes the estimator computable.
\eremk
\end{remark}
{Similarly to \eqref{eq:mlmc_vem_u}, we define the MLMC-VE estimator for~$\E[\Q(u)]$ as
\begin{gather}
    \label{eq:mlmc_vem_QoI}
    \mlmc[\mathcal{Q}(u)]= \sum_{\ell=1}^L \mc{M_\ell}[\mathcal{Q}(\nu_\ell)] 
    = \sum_{\ell=1}^L \frac{1}{M_\ell} \sum_{i=1}^{M_\ell }\mathcal{Q}(\nu^{(i)}_\ell) \in L^2(\Omega),
\end{gather}
where
$$
\nu_\ell=\Pi^0_\ell u_\ell
    -\Pi^0_{\ell-1} u_{\ell-1}, \qquad
\nu^{(i)}_\ell=\Pi^0_\ell u^{(i)}_\ell
    -\Pi^0_{\ell-1} u^{(i)}_{\ell-1}.
$$
}

In the next theorem, we establish error estimates for the MLMC-VE estimators~\eqref{eq:mlmc_vem_u} and~\eqref{eq:mlmc_vem_QoI}.

\begin{theorem} [Error estimate for the MLMC-VE estimators]
    \label{thm:mlmc_u_conv}
    Assume that $u\in \todo{L^2(\Omega,H^{s+1}(D))}$\todo{, $s\ge 1$}. 
    Then, we have\todo{, with~$\mus=\min\{s,p\}$} 
    \begin{gather}
        \label{eq:mlmc_u_conv}
        \begin{aligned}
        &\|\E[u]-\mlmc[u]\|_{L^2(\Omega,{\VbTL})}\\
        &\lesssim \todo{\frac{s^\star(p)}{s_\star(p)}}\norm{\frac{a_{max}}{a_{min}}}{L^2(\Omega)}{p^{-s}} \Bigg(h_L^{\todo{\mus}} + M_1^{-1/2}h_1^{\todo{\mus}}+
        \sum_{\ell=2}^L M_\ell^{-1/2} (h_\ell^{\todo{\mus}}+2h_{\ell-1}^{\todo{\mus}})\Bigg)\|u\|_{L^{2}(\Omega,H^{s+1}(D))}.
        \end{aligned}
    \end{gather}
    {Moreover, under the assumptions of Theorem \ref{thm:qoi_apriori}, the following error bound holds\todo{, with~$\mu=\min\{s,t,p\}$}:
    \begin{gather}
        \label{eq:mlmc_qoi_conv}
        \begin{aligned}
        &\|\E[\mathcal{Q}(u)]-\mlmc[\mathcal{Q}(u)]\|_{L^2(\Omega)}\\
        &\lesssim \todo{\frac{s^\star(p)}{s_\star(p)}}\norm{\Big(\frac{a_{max}}{a_{min}}\Big)^{\todo{2}} a_{min}^{-1}}{L^2(\Omega)} p^{-\todo{2\,\todo{\min\{s,t-1\}}}} 
        \Bigg(h_L^{2\todo{\mu}} + M_1^{-1/2} h_1^{2\todo{\mu}}+
        \sum_{\ell=2}^L M_\ell^{-1/2} (h_\ell^{2\todo{\mu}}+h_{\ell-1}^{2\todo{\mu}})\Bigg)\\
        &\qquad
        \left(\norm{q}{H^{\todo{t}-1}(D)}^2+\norm{f}{H^{\todo{s}-1}(D)}^2\right).
        \end{aligned}
    \end{gather}}
\end{theorem}

\begin{proof}
{We start proving \eqref{eq:mlmc_u_conv}.} By the triangular inequality, we get
\begin{gather*}
\begin{aligned}
&\|\E[u]-\mlmc[u]\|_{L^2(\Omega,{\VbTL})}\\
&\qquad \leq \underbrace{\|\E[u]-\E[\Pi^\nabla_Lu_L]\|_{L^2(\Omega,{\VbTL})}}_{(I)}
+ \underbrace{\|\E[\Pi^\nabla_Lu_L]-\mlmc[u]\|_{L^2(\Omega,{\VbTL})}}_{(II)}.
\end{aligned}
\end{gather*}
For~$(I)$, proceeding as in the proof of Proposition~\ref{prop:mcvem}, we obtain
\[
(I)\lesssim \todo{\frac{s^\star(p)}{s_\star(p)}}\norm{\frac{a_{max}}{a_{min}}}{L^2(\Omega)} \frac{h_L^{\todo{\mus}}}{p^s}  \|u\|_{L^2(\Omega,H^{{s}+1}(D))}.
\]
We focus now on the second term $(II)$. Expressing $\Pi^\nabla_L u_L$ as the telescopic sum
\[
\Pi^\nabla_L u_L = \sum_{\ell=1}^L {(\Pi^\nabla_\ell u_\ell -\Pi^\nabla_{\ell-1} u_{\ell-1})=\sum_{\ell=1}^L w_\ell},
\]
the linearity of the expectation, together with \eqref{eq:mlmc_vem_u} and~\eqref{eq:mc_error_u}, yields
\begin{align*}
\nonumber
(II) =
\Bigg\|\sum_{\ell=1}^L (\E[w_\ell]-\mc{M_\ell}[w_\ell])\Bigg\|_{L^2(\Omega,{\VbTL})}
&\leq \sum_{\ell=1}^L \|\E[w_\ell]-\mc{M_\ell}[w_\ell]\|_{L^2(\Omega,{\VbTL})}\\
\label{eq:IIa}
&\leq \sum_{\ell=1}^L M_\ell^{-1/2}\|w_\ell\|_{L^2(\Omega,{\VbTL})},
\end{align*}
where the last inequality follows from~\eqref{eq:mc_error_u}.
We estimate $\|w_\ell\|_{L^2(\Omega,{\VbTL})}$ as follows:
\begin{align*}
    \nonumber
    \|w_\ell\|_{L^2(\Omega,{\VbTL})}
    & = 
    \|\Pi^\nabla_\ell (u_\ell -\Pi^\nabla_{\ell-1} u_{\ell-1})\|_{L^2(\Omega,{\VbTL})}\\
    & \lesssim\left\|
    u_\ell -\Pi^\nabla_{\ell-1} u_{\ell-1}\right\|_{L^2(\Omega,{\VbTL})}\\
& \lesssim\left\|
    u_\ell -u\right\|_{L^2(\Omega,{\VbTL})}+\left\|
    u -\Pi^\nabla_{\ell-1} u\right\|_{L^2(\Omega,{\VbTL})}\\
    &\qquad +\left\|
    \Pi^\nabla_{\ell-1} (u -u_{\ell-1})\right\|_{L^2(\Omega,{\VbTL})}\\
    &\lesssim \todo{\frac{s^\star(p)}{s_\star(p)}}\norm{\frac{a_{max}}{a_{min}}}{L^2(\Omega)}\frac{(h_\ell^{\mus}+2 h_{\ell-1}^{\mus})}{{{p^s}}} \|u\|_{L^{2}(\Omega,H^{s+1}(D))},
\end{align*}
where we used the identity $ \Pi^\nabla_{\ell} (\Pi^\nabla_{\ell-1} u_{\ell-1})= \Pi^\nabla_{\ell-1} u_{\ell-1}$, 
which is a consequence of the fact that the meshes are nested, cf. \eqref{eq:nested_meshes}, 
the triangle inequality, the stability and approximation properties of the projection operators, as well as the error estimate~\eqref{vem:estimate}. Then, we obtain
\[
(II)\lesssim \todo{\frac{s^\star(p)}{s_\star(p)}}\norm{\frac{a_{max}}{a_{min}}}{L^2(\Omega)}\sum_{\ell=1}^L (M_\ell)^{-1/2} \frac{(h_\ell^{\mus}+2 h_{\ell-1}^{\mus})}{{{p^s}}} \|u\|_{L^{2}(\Omega,H^{s+1}(D))},
\]
which concludes the proof of~\eqref{eq:mlmc_u_conv}.

{The proof of~\eqref{eq:mlmc_qoi_conv} works similarly. First, we sum and subtract $\E[\mathcal{Q}(u)]$ and use the triangular inequality:
\begin{gather*}
\begin{aligned}
&\|\E[\mathcal{Q}(u)]-\mlmc[\mathcal{Q}(u)]\|_{L^2(\Omega,V)}\\
&\quad\leq \underbrace{\|\E[\mathcal{Q}(u)]-\E[\mathcal{Q}(\Pi^0_L u_L)]\|_{L^2(\Omega)}}_{(I)}
+ \underbrace{\|\E[\mathcal{Q}(\Pi^0_Lu_L)]-\mlmc[\mathcal{Q}(u)]\|_{L^2(\Omega)}}_{(II)}.
\end{aligned}
\end{gather*}
The term $(I)$ is the same as in the proof of Proposition~\ref{prop:mcvem}, hence there holds
\[
(I) \lesssim \todo{\frac{s^\star(p)}{s_\star(p)}}\E\Big[\Big(\frac{a_{max}}{a_{min}}\Big)^{\todo{2}} a_{min}^{-1}\Big] \frac{h_L^{2\todo{\mu}}}{p^{\todo{2\,\todo{\min\{s,t-1\}}}}}  \left(\norm{q}{H^{\todo{t}-1}(D)}^2+\norm{f}{H^{\todo{s}-1}(D)}^2\right).
\]
To estimate $(II)$, we express $\Pi^0_L u_L$ as telescopic sum 
\[
\Pi^0_L u_L = \sum_{\ell=1}^L (\Pi^0_\ell u_\ell -\Pi^0_{\ell-1} u_{\ell-1})=\sum_{\ell=1}^L \nu_\ell,
\]
and we use the linearity of the expectation, the linearity of $\mathcal{Q}$, together with~\eqref{eq:mlmc_vem_QoI} and~\eqref{eq:mc_error_u}:
\[
(II)\leq \sum_{\ell=1}^L M_\ell^{-1/2}\|\mathcal{Q}(\nu_\ell)\|_{L^2(\Omega)}.
\]
Finally, we estimate $\|\mathcal{Q}(\nu_\ell)\|_{L^2(\Omega)}$ as follows:
\begin{align}
    \nonumber
    \|\mathcal{Q}(\nu_\ell)\|_{L^2(\Omega)}
    &=\|\mathcal{Q}(\Pi^0_\ell u_\ell -\Pi^0_{\ell-1} u_{\ell-1})\|_{L^2(\Omega)} 
    \\\nonumber
    &\leq
    \|\mathcal{Q}(\Pi^0_\ell u_\ell - u)\|_{L^2(\Omega)}+
    \|\mathcal{Q}(u -\Pi^0_{\ell-1} u_{\ell-1})\|_{L^2(\Omega)}\\\nonumber
    &\lesssim
    \todo{\frac{s^\star(p)}{s_\star(p)}}\norm{\Big(\frac{a_{max}}{a_{min}}\Big)^{\todo{2}} a_{min}^{-1}}{L^2(\Omega)}\frac{h_\ell^{2\todo{\mu}}+h_{\ell-1}^{2\todo{\mu}}}{{p^{\todo{2\,\todo{\min\{s,t-1\}}}}}}\left(\norm{q}{H^{\todo{t}-1}(D)}^2+\norm{f}{H^{\todo{s}-1}(D)}^2\right),
\end{align}
where we used~\eqref{eq:BQ}. Then, we obtain
\[
(II)\lesssim \todo{\frac{s^\star(p)}{s_\star(p)}}\norm{\Big(\frac{a_{max}}{a_{min}}\Big)^{\todo{2}} a_{min}^{-1}}{L^2(\Omega)}
\sum_{\ell=1}^L (M_\ell)^{-1/2} \frac{h_\ell^{2\todo{\mu}}+h_{\ell-1}^{2\todo{\mu}}}{{p^{\todo{2\,\todo{\min\{s,t-1\}}}}}}\left(\norm{q}{H^{\todo{t}-1}(D)}^2+\norm{f}{H^{\todo{s}-1}(D)}^2\right),
\]
which concludes the proof of~\eqref{eq:mlmc_qoi_conv}.
}
\end{proof}
{ 
\begin{remark}
Employing the MC-VE estimator defined in  \eqref{EM:new_def} in the construction of $E^L[u]$ (cf. \eqref{eq:mlmc_vem_u}), one could easily adapt the above proof and obtain an estimate for the error $\|\E[u]-\mlmc[u]\|_{L^2(\Omega,L^2(D))}$.
\eremk
\end{remark}}

{The previous theorem gives the error estimate for the MLMC-VE method for any distribution of number of samples per level $\{M_\ell\}_{\ell=1}^L$. The following result, being a multilevel version of Remark~\ref{rem:mc_Mopt}, } specifies how to choose the 
number of samples $M_\ell$ 
at each mesh level~$\ell$ to achieve an overall error of order~$\mathcal{O}(h_L^{\todo{\mus}}\, p^{-s}\frac{s^\star(p)}{s_\star(p)})$ for the MLMC-VE estimator.

\begin{theorem}
\label{thm:mlmc_Mopt}
Consider the same assumptions as in Theorem~\ref{thm:mlmc_u_conv}. Moreover, assume, for simplicity, that $h_{\ell-1}=2 h_\ell$ for all $\ell=1,\ldots,L$. Then, the MLMC-VE estimator $\mlmc[u]$ with 
\begin{equation}
\label{eq:Mopt-h}
M_\ell \sim 
\bigg(h_\ell^{\todo{\mus}}\,h_L^{-{\todo{\mus}}}\ell^{(1+\varepsilon)}\bigg)^2
\end{equation}
samples (with $\varepsilon>0$ arbitrarily small) on the mesh level $\ell$ admits the error bound
\begin{align}
    \label{eq:mlmc_u_Mopt}
    \|\E[u]-\mlmc[u]\|_{L^2(\Omega,{\VbTL})}
    & \lesssim \frac{s^\star(p)}{s_\star(p)}\frac{h_L^{\todo{\mus}}}{{{p^s}}} \|u\|_{L^2(\Omega,H^{s+1}(D))}.
\end{align}
Similarly, for the MLMC-VE estimator $\mlmc[\Q(u)]$, with 
\[
M_\ell=\bigg(h_\ell^{{\todo{2\mu}}}\,h_L^{-{\todo{2\mu}}}\ell^{(1+\varepsilon)}\bigg)^2
\]
samples on the mesh level $\ell$, the following bound holds:
\begin{align*}
    \|\E[\Q(u)]-\mlmc[\Q(u)]\|_{L^2(\Omega)}
   & \lesssim \todo{\frac{s^\star(p)}{s_\star(p)}} 
        \frac{h_L^{2\todo{\mu}}}{p^{\todo{2\,\todo{\min\{s,t-1\}}}} } \left(\norm{q}{H^{\todo{t}-1}(D)}^2+\norm{f}{H^{\todo{s}-1}(D)}^2\right).
\end{align*}
\end{theorem}
\begin{proof}
With $h_{\ell-1}=2 h_\ell$ and $M_\ell$ as in~\eqref{eq:Mopt-h},
we get
\begin{gather*}
\begin{aligned}
    & {M_1^{-1/2}h_1^{\todo{\mus}}+
        \sum_{\ell=2}^L M_\ell^{-1/2} (h_\ell^{\todo{\mus}}+2h_{\ell-1}^{\todo{\mus}})
    = M_1^{-1/2}h_1^{\todo{\mus}}+\sum_{\ell=2}^L M_\ell^{-1/2} (h_\ell^{\todo{\mus}}+2\, 2^{\todo{\mus}} h_{\ell}^{\todo{\mus}})}\\
    &{= M_1^{-1/2}h_1^{\todo{\mus}} + (1+2^{{\todo{\mus}}+1})\sum_{\ell=2}^L M_\ell^{-1/2} h_{\ell}^{\todo{\mus}}
    \sim h_1^{-{\todo{\mus}}}h_L^{\todo{\mus}} h_1^{\todo{\mus}} + (1+2^{{\todo{\mus}}+1})\sum_{\ell=2}^L
    \frac{h_L^{\todo{\mus}}}{h_\ell^{\todo{\mus}}}\, \ell^{-(1+\varepsilon)}h_\ell^{\todo{\mus}}}\\
    &{< (1+2^{{\todo{\mus}}+1})\, h_L^{\todo{\mus}} \sum_{\ell=1}^L \ell^{-(1+\varepsilon)}
    = C(\varepsilon)(1+2^{{\todo{\mus}}+1})\, h_L^{\todo{\mus}},}
\end{aligned}    
\end{gather*}
where $C(\varepsilon)$ is a positive constant depending on $\varepsilon$ (more precisely, $C(\varepsilon)=\zeta(1+\varepsilon)$, with $\zeta$ the Riemann zeta function). Inserting this into~\eqref{eq:mlmc_u_conv} gives~\eqref{eq:mlmc_u_Mopt}. The result for the QoI can be proved with a similar argument.
\end{proof}

A practical choice for the sample sizes, which avoids requiring knowledge of the solution’s regularity, is discussed in Section~\ref{sec:implementation_choices} below.
\begin{remark}
For the MLMC-VE estimator, we write
\begin{gather}
\label{eq:bias-variance}
\begin{aligned}
    \norm{\E[u]-\mlmc[u]}{L^2(\Omega,H^1(\T_L))}^2
    &= \norm{\E[u-\Pi^\nabla_L u_L]}{L^2(\Omega,H^1(\T_L))}^2
    + \mathbb{V}[\mlmc{[u]}]\\
    &= \norm{\E[u-\todo{\Pi^\nabla_L}u_L]}{L^2(\Omega,H^1(\T_L))}^2
    + \sum_{\ell=1}^L\frac{\mathbb{V}[\Pi^\nabla_\ell u_{\ell}-\Pi^\nabla_{\ell-1} u_{\ell-1}]}{M_\ell},
\end{aligned}
\end{gather}
where $\mathbb{V}[\mlmc{[u]}]\coloneqq \norm{\mlmc{[u]}-\E[\mlmc{[u]}]}{L^2(\Omega,H^1(\T_L))}^2$. The first identity follows from~$\E[\mlmc{[u]}]=\E[\Pi^\nabla_L u_L]$, which implies
 $$\mathbb{V}[\mlmc{[u]}] 
=\norm{\mlmc{[u]}-\E[\Pi^\nabla_L u_L]}{L^2(\Omega,H^1(\T_L))}^2,$$ while the second identity follows from
$$\mathbb{V}[\mc{M_\ell}[\Pi^\nabla_\ell u_{\ell}-\Pi^\nabla_{\ell-1} u_{\ell-1}]]=\todo{\frac{1}{M_\ell}}\mathbb{V}[\Pi^\nabla_\ell u_{\ell}-\Pi^\nabla_{\ell-1} u_{\ell-1}].$$
The error split in~\eqref{eq:bias-variance} is known as bias--variance decomposition. The first term is the bias component of the error, which is due to the VE discretization error at the finest level~$L$.
The second term is the variance component of the error, which is due to the sampling error.
Assume that one wants to define the estimator $\mlmc[u]$ by selecting the finest level~$L$ and the number of samples~$M_\ell$ per level such that 
\[
\norm{\E[u]-\mlmc[u]}{L^2(\Omega,H^1(\T_L))}\le \varepsilon,
\]
for a prescribed tolerance~$\varepsilon>0$. To achieve this, the tolerance is typically distributed between the bias and the variance: 
\begin{equation*}
\varepsilon^2= \varepsilon^2_{\mathrm{bias}} + \varepsilon^2_{\mathrm{var}},  
\end{equation*}
where, for $\theta\in(0,1)$, $\varepsilon^2_{\mathrm{bias}}=\theta\varepsilon^2$ controls the space discretization error and $\varepsilon_{\mathrm{var}}^2=(1-\theta)\varepsilon^2$ cotrols the sampling error.
An extensive discussion on the optimal choice of the splitting parameter can be found in~\cite{HajiAli-Nobile-vonSchwerin-Tempone}. The most common choice is $\theta=1/2$, yielding an equal splitting of the total tolerance between the bias and variance components. Following the bias--variance decomposition, the finest level~$L$ is chosen so that the discretization error is below~$\varepsilon_{\mathrm{bias}}$, and the number of samples~$M_\ell$ at each level~$\ell$ is selected so as to minimize the variance contribution to the total error for a given computational cost. This leads to 
$$M_\ell \propto 
\sqrt{\frac{\mathbb{V}[\Pi^\nabla_\ell u_{\ell}-\Pi^\nabla_{\ell-1} u_{\ell-1}]}{\mathcal C_\ell}},$$ 
where $\mathcal{C}_\ell$ denotes the cost of computing a single realization of $\Pi^\nabla_\ell u_{\ell}-\Pi^\nabla_{\ell-1} u_{\ell-1}$ (see, e.g.,~\cite{Cliffe-Giles-Scheichl-Teckentrup}).
\eremk
\end{remark}

\section{Numerical results}
\label{sec:numerics}

In this section, we first verify the convergence estimate of the VE error for a linear QoI in the deterministic case, as established in Section~\ref{sec:qoi_apriori}. We then proceed to validate 
\subsection{Verification of the estimates of the VE error for the QoI}
\label{sec:verification_test_qoi}

Let us consider problem~\eqref{eq:strong_pde_det} on $D=(0,1)\times(0,1)\subset\R^2$ with~$\alpha=1$ and~$f(x_1,x_2)=32\pi^2u_{ex}(x_1,x_2)$, whose analytical solution is $u_{ex}(x_1,x_2)=\sin(4\pi x_1)\sin(4\pi x_2)$. We use the sequence of non-nested meshes represented in Figure~\ref{fig:MeshesVoronoi}. Taking $q(x)=1$ in the definition of the quantity of interest~\eqref{eq:q}, 
we get $\Q(u_{ex})=0$. We compute the approximation $\Q_h(u_h)$ using the VE method of order 
$p=1,2,3$. The resulting errors represented in Figure~\ref{fig:qoi1_convergence} demonstrate that $|\Q(u_{ex})-\Q_h(u_h)|=\mathcal{O}(h^{2p})$, as proved in Theorem~\ref{thm:qoi_apriori}. 
\begin{figure}[!htbp]
\centering
   \begin{subfigure}[b]{0.33\textwidth} 
          \includegraphics[width=\textwidth]{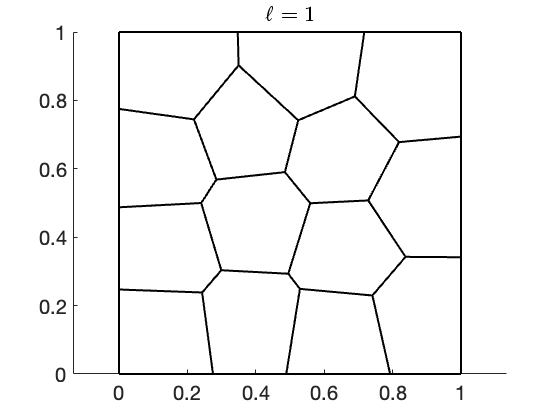}
    \end{subfigure}%
    \begin{subfigure}[b]{0.33\textwidth} 
          \includegraphics[width=\textwidth]{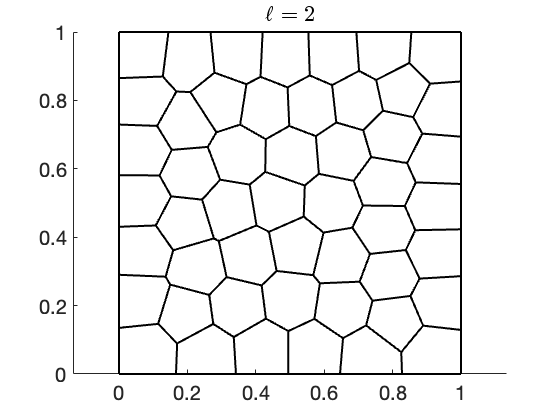}
    \end{subfigure}%
    \begin{subfigure}[b]{0.33\textwidth} 
          \includegraphics[width=\textwidth]{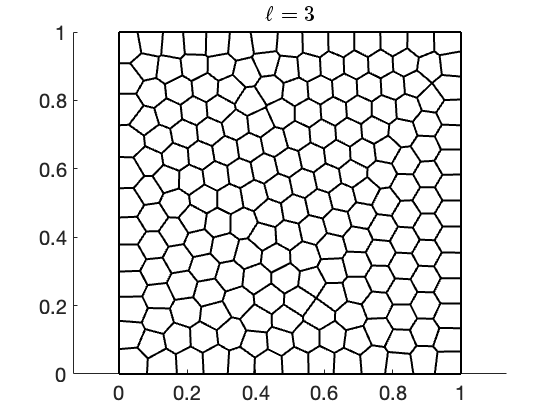}
    \end{subfigure}%
    \\
    \begin{subfigure}[b]{0.33\textwidth} 
          \includegraphics[width=\textwidth]{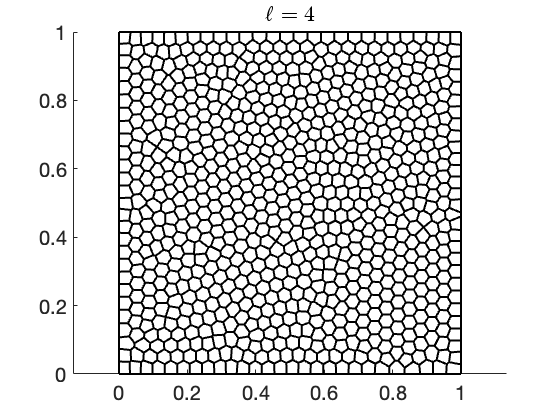}
    \end{subfigure}%
    \begin{subfigure}[b]{0.33\textwidth} 
          \includegraphics[width=\textwidth]{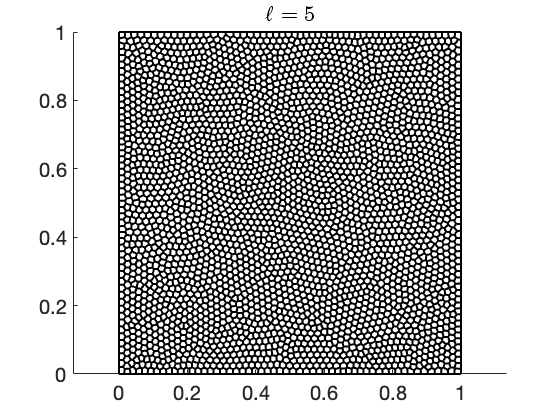}
    \end{subfigure}%
    \begin{subfigure}[b]{0.34\textwidth} 
          \includegraphics[width=\textwidth]{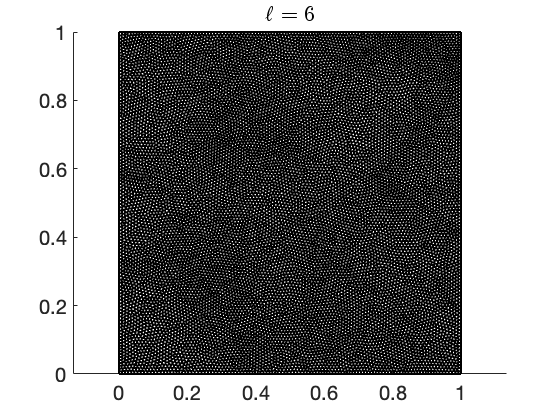}
    \end{subfigure}%
    \caption{Sequence of non-nested Voronoi meshes used in the numerical tests of Section \ref{sec:verification_test_qoi}.}
    \label{fig:MeshesVoronoi}
\end{figure}
\begin{figure}[!htbp]
    \centering
    \begin{subfigure}[b]{0.46\textwidth} 
          \resizebox{\textwidth}{!}{\pgfplotsset{
  log x ticks with fixed point/.style={
      xticklabel={
        \pgfkeys{/pgf/fpu=true}
        \pgfmathparse{exp(\tick)}%
        \pgfmathprintnumber[fixed  zerofill, precision=2]{\pgfmathresult}
        \pgfkeys{/pgf/fpu=false}
      }
  }
}

\begin{tikzpicture}

\begin{axis}[%
width=3.875in,
height=2.36in,
at={(2.6in,1.099in)},
scale only axis,
xmode=log,
xmin=0.005,
xmax=0.5,
xminorticks=true,
xlabel = {$h$ [-]},
ylabel = {$|\Q(u_{ex})-\Q_h(u_h)|$},
ymode=log,
ymin=1e-13,
ymax=5e-1,
yminorticks=true,
axis background/.style={fill=white},
title style={font=\bfseries},
xmajorgrids,
ymajorgrids,
legend columns=3,
legend style={
    at={(0.98,0.1)},
    anchor=east}
]

\addplot [color=red, mark=square,mark size=2.9pt, line width=2.0pt]
  table[row sep=crcr]{%
4.1654e-01   2.4984e-02\\
2.1521e-01   5.0135e-03\\
1.0952e-01   1.2817e-03\\
5.2280e-02   6.0326e-04\\
2.7932e-02   8.1266e-05\\
1.3089e-02   4.4624e-06\\};
\addlegendentry{$p=1$}

\addplot [color=blue, mark=o,mark size=2.9pt, line width=2.0pt]
  table[row sep=crcr]{%
4.1654e-01   5.6507e-03\\
2.1521e-01   2.7160e-04\\
1.0952e-01   6.8396e-06\\
5.2280e-02   8.7367e-07\\
2.7932e-02   2.1745e-08\\
1.3089e-02   1.8206e-09\\};
\addlegendentry{$p=2$}

\addplot [color=magenta, mark=triangle,mark size=2.9pt, line width=2.0pt]
  table[row sep=crcr]{%
4.1654e-01   7.8437e-04\\
2.1521e-01   5.5233e-05\\
1.0952e-01   5.7093e-07\\
5.2280e-02   1.0607e-08\\
2.7932e-02   1.9411e-10\\
1.3089e-02   8.1165e-12\\};
\addlegendentry{$p=3$}

\addplot [color=red, mark=square,mark size=2.9pt, mark options=solid, style = dashed,line width=2.0pt]
  table[row sep=crcr]{%
4.1654e-01   8.6752e-03\\
2.1521e-01   2.3158e-03\\
1.0952e-01   5.9975e-04\\
5.2280e-02   1.3666e-04\\
2.7932e-02   3.9011e-05\\
1.3089e-02   8.5664e-06\\};
\addlegendentry{$h^2$}

\addplot [color=blue, mark=o,mark size=2.9pt, mark options=solid, style = dashed,line width=2.0pt]
  table[row sep=crcr]{%
4.1654e-01   2.4083e-03\\
2.1521e-01   1.7161e-04\\
1.0952e-01   1.1510e-05\\
5.2280e-02   5.9762e-07\\
2.7932e-02   4.8699e-08\\
1.3089e-02   2.3483e-09\\};
\addlegendentry{$h^4$}

\addplot [color=magenta, mark=triangle,mark size=2.9pt, mark options=solid, style = dashed,line width=2.0pt]
  table[row sep=crcr]{%
4.1654e-01   5.2230e-03\\
2.1521e-01   9.9356e-05\\
1.0952e-01   1.7258e-06\\
5.2280e-02   2.0418e-08\\
2.7932e-02   4.7495e-10\\
1.3089e-02   5.0290e-12\\};
\addlegendentry{$h^6$}

\end{axis}
\end{tikzpicture}%
}
    \end{subfigure}%
    \caption{VE error for the QoI, $|\Q(u_{ex})-\Q_h(u_h)|$, for 
    $p=1,2,3$.}
    \label{fig:qoi1_convergence}
\end{figure}
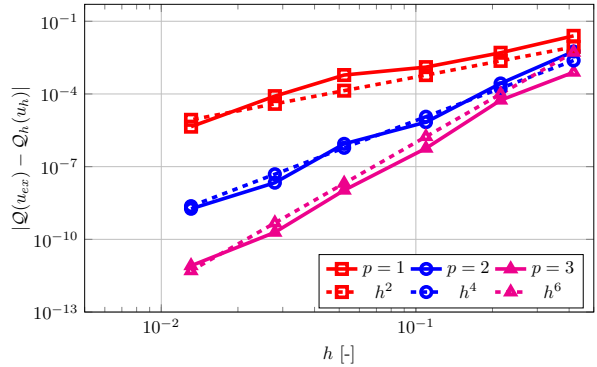

\subsection{Practical selection of MC and MLMC sample sizes}
\label{sec:implementation_choices}

We discuss the practical selection of MC and MLMC sample sizes, taking into account the comparison between the MC-VE and MLMC-VE methods.
Since the smoothness of the solution~$u$ may not be known a priori, we adopt the following choice of the number~$M$ of MC samples:
\begin{gather}
\label{eq:M_practical}
M=\left\{\begin{array}{ll}
     p^{2p}h^{-2p}, & \textrm{for }\mc{M}[u], \\
     p^{4p}h^{-4p}, & \textrm{for }\mc{M}[\Q(u)]. \\
\end{array}\right.
\end{gather}
Under the assumptions in the previous sections, this choice ensures the error bounds
\begin{gather*}
    \|\E[u]-\mc{M}[\Pi^\nabla_{h,p}u_{h,p}]\|_{L^2(\Omega,\VbTh)} 
        =\mathcal{O}(h^{\mu_s}),\\
     \|\E[\mathcal Q(u)]-\mc{M}[\mathcal{Q}(\Pi^0_{h,p}u_{h,p})]\|_{L^2(\Omega)} 
    =\mathcal{O}(h^{2\mu}),     
\end{gather*}
where we recall the defintions~$\mu_s=\min\{s,p\}$ and~$\mu=\min\{s,t,p\}$.

Similarly, the practical choice we adopt for the number of MLMC samples~$M_\ell$ at the level~$\ell$ is 
\begin{gather}
\label{eq:Mell_practical}
M_\ell=\left\{\begin{array}{ll}
     \bigg(h_\ell^{p}\,h_L^{-p}\ell^{(1+\varepsilon)}\bigg)^2, & \textrm{for }\mc{M}[u], \\
     \bigg(h_\ell^{2p}\,h_L^{-2p}\ell^{(1+\varepsilon)}\bigg)^2, & \textrm{for }\mc{M}[\Q(u)]. \\
\end{array}\right.
\end{gather}
ensuring the error bounds
\begin{gather*}
    \|\E[u]-\mlmc[u]\|_{L^2(\Omega,{\VbTL})}
    =\mathcal{O}(h_L^{\mu_s}),\\
    \|\E[\mathcal{Q}(u)]-\mlmc[\mathcal{Q}(u)]\|_{L^2(\Omega)}
        =\mathcal{O}(h_L^{2\mu}).
\end{gather*}
If the solution has high regularity, then the order of convergence $\mu_s$ of the MC-VE and MLMC-VE errors on $\E[u]$ is equal to $p$, and the order of convergence $2\mu$ of the MC-VE and MLMC-VE errors on $\E[\Q(u)]$ is equal to $2p$, provided that the function~$q$ defining the QoI is sufficiently smooth. However, if the solution is expected to have low regularity (and~$q$ has low regularity, in case of~$\E[\Q(u)]$), the use of high 
orders~$p$ is typically not beneficial. Therefore, we limit the following discussion to the case~$p\leq s$ (and~$p\le t$, in case of~$\E[\Q(u)]$).
Given the practical choice~\eqref{eq:Mell_practical}, we estimate the computational complexity of the algorithm for the computation of the MLMC-VE estimator of $\E[u]$ as follows.
\begin{theorem}
\label{thm:h-mlmc-cost}
Let $N_\ell$ denote the number of degrees of freedom of the VE space of degree $p$ on the mesh $\mathcal{T}_\ell$, and assume that $N_\ell=\mathcal O(2^{2\ell} p^2)$ (this holds, for example, for Cartesian meshes).
Moreover, assume that, at each level~$\ell$, the VE equations for each sample $u_\ell^{(i)}$ are solved employing a multigrid method, and that~$M_\ell$ is chosen according to~\eqref{eq:Mell_practical}.
Then, the complexity ${\mathbb{C}}_{MLMC}(L)$ for computing the MLMC-VE estimator $\mlmc[u]$ ensuring an accuracy of order~$\mathcal{O}(h_L^{p})$
satisfies 
\begin{gather}
\label{eq:h-mlmc-cost}
    \mathbb{C}_{MLMC}(L)\lesssim\left\{\begin{array}{cc}
        N_L L^{3+2\varepsilon}, & \text{if } p=1, \\
        N_L^p p^{2(1-p)}, & \text{if }p\geq 2.
    \end{array}\right.
\end{gather}
\end{theorem}

\begin{proof}
Choosing $M_\ell$ as in~\eqref{eq:Mell_practical}, using~$h_\ell\sim 2^{-\ell}$ and assuming that the complexity of the multigrid algorithm at level $\ell$ is proportional to $N_\ell$, then there holds
\begin{gather}
    \begin{aligned}
        {\mathbb{C}_{MLMC}(L)}
        &\lesssim\sum_{\ell=1}^L M_\ell N_\ell
        \sim \sum_{\ell=1}^L \bigg( h_\ell^p h_L^{-p} \ell^{(1+\varepsilon)} \bigg)^2 N_\ell
        \sim \sum_{\ell=1}^L \bigg( 2^{p(L-\ell)} \ell^{(1+\varepsilon)} \bigg)^2 2^{2\ell} p^2 \\ &
        = p^2 2^{2L}\sum_{\ell=1}^L 2^{2p(L-\ell)} 2^{2(\ell-L)}\ell^{2(1+\varepsilon)}
        \label{eq:sum_l}
        \sim N_L\sum_{\ell=1}^L 2^{(\ell-L)(2-2p)}\ell^{2(1+\varepsilon)}
    \end{aligned}
\end{gather}
We now estimate the sum in \eqref{eq:sum_l}.
\begin{itemize}
    \item For $p=1$, there holds
    \[
    \sum_{\ell=1}^L 2^{(\ell-L)(2-2p)}\ell^{2(1+\varepsilon)}
    = \sum_{\ell=1}^L \ell^{2(1+\varepsilon)} 
    = \sum_{\ell'=0}^{L-1} (L-\ell')^{2(1+\varepsilon)} 
    = \mathcal{O}(L^{3+2\varepsilon}).
    \]
    \item For $p\geq 2$, then $2-2p$ is negative and we call it $-k$ with $k>0$. Then, there holds
    \begin{gather*}
        \begin{aligned}
            &\sum_{\ell=1}^L 2^{(\ell-L)(2-2p)}\ell^{2(1+\varepsilon)}
            = \sum_{\ell=1}^L 2^{k(L-\ell)}\ell^{2(1+\varepsilon)}
            = \sum_{\ell'=0}^{L-1} (2^{k})^{\ell'} (L-\ell')^{2(1+\varepsilon)}\\
            &= \mathcal{O}(2^{kL})
            =\mathcal{O}(2^{2L(p-1)})
            = \mathcal{O}((N_L p^{-2})^{(p-1)}).
        \end{aligned}
    \end{gather*}
\end{itemize}
Hence, \eqref{eq:h-mlmc-cost} follows.
\end{proof}

From Theorem~\ref{thm:h-mlmc-cost}, the computational advantage of the MLMC-VE method with respect to the MC-VE method on the finest mesh~$\mathcal{T}_L$ is evident. Indeed, choosing~$M$ as in~\eqref{eq:M_practical} and using $N_L=\mathcal O(2^{2L} p^2)$, the complexity~$\mathbb{C}_{MC}(L)$ for computing the MC-VE estimator of $\E[u]$ to ensure an accuracy of order~$\mathcal{O}(h_L^{p})$ 
is estimated as
\begin{eqnarray}
    {\mathbb{C}_{MC}}(L)=M N_L
    = \mathcal{O}(p^{2p} h_L^{-2p}) N_L
    = \mathcal{O}(N_L^{p+1})
    \nonumber
\end{eqnarray}
Comparing~\eqref{eq:h-mlmc-cost} in Theorem~\ref{thm:h-mlmc-cost} with the above result indicates that the MLMC-VE method achieves the same accuracy of MC-VE with a substantial reduction of the complexity in the considered case $p\leq s$.

For the QoI, following similar steps as above, we can show that the practical choice of the number of samples~$M_\ell$ in~\eqref{eq:Mell_practical} for the computation of the MLMC-VE estimator of $\E[\Q(u)]$ results into a computational complexity  $\mathcal{O}(N_L^{2p} p^{2(1-2p)})$, independently of whether~$p=1$ or~$p\ge 2$, which is substantially smaller than the complexity of computing the MC-VE estimator $\mc{M}[\Q(u)]$ on the finest mesh~$\mathcal{T}_L$ with the practical choice of $M$ as in~\eqref{eq:M_practical}, which is~$\mathcal O(N_L^{2p+1})$.

\subsection{Verification of the error estimates for the MLMC-VE method}
\label{sec:verification_test_hMLMC}

Let us consider problem~\eqref{eq:strong_pde} defined on $D=(0,1)\times(0,1)\subset\R^2$ with constant forcing term $f(\omega,x_1,x_2)=1$ and random diffusion coefficient $a(\omega,x_1,x_2)= 5 + x_1 + x_2 + \big(\frac{8}{\pi^2}\big)^{2.5}Y(\omega)\sin\big(\frac{\pi}{4}(x_1+1)\big) \sin(\frac{\pi}{4}(x_2+1)\big)$, where $Y\sim\mathcal{U}([-1,1])$ is a uniformly distributed random variable. This example is taken from \cite[Section 6]{BarthSchwabZollinger2011}.
With the aim of approximating $\E[u]$ by means of the MLMC-VE estimator, 
we consider a sequence of nested meshes made of square elements with side length~$1/2^\ell$, ($h_\ell=\frac{\sqrt{2}}{2^{\ell+1}}$), for $\ell=1,\ldots,6$.
In Figure~\ref{fig:NumberSamples}, we show the number of samples~$M_\ell$ per mesh level~$\ell$ required to achieve $\|\E[u]-\mlmc[u]\|_{L^2(\Omega,{\Vb})}=\mathcal{O}(h_L^{p})$
for~$L=6$, as predicted by Theorem~\ref{thm:h-mlmc-cost}. 
We recall that~$\mlmc[u]=\sum_{\ell=1}^L \mc{M_\ell}[w_\ell]$, see~\eqref{eq:mlmc_vem_u}. In particular, Figure~\ref{fig:NumberSamples} (left) displays in semilog scale the number of samples per level
\[
\bigg\{M_\ell=\bigg(h_\ell^p\,h_L^{-p}\ell^{(1+\varepsilon)}\bigg)^2,\, \ell=1,\ldots,6\bigg\}
\] 
needed to compute the MLMC estimator~$\mlmc[u]$ 
with $L=6$ and $\varepsilon=1e-10$, for different values of the order~$p=1,2,3$. In Figure~\ref{fig:NumberSamples} (right), we plot in semilog scale the total number of samples $\sum_{\ell=1}^L M_\ell$ versus the maximum level~$L$ for $L=1,\ldots,6$.
As expected, the number of samples increases as the maximum level $L$ increases, and the rate of increase grows with the 
order~$p$.  
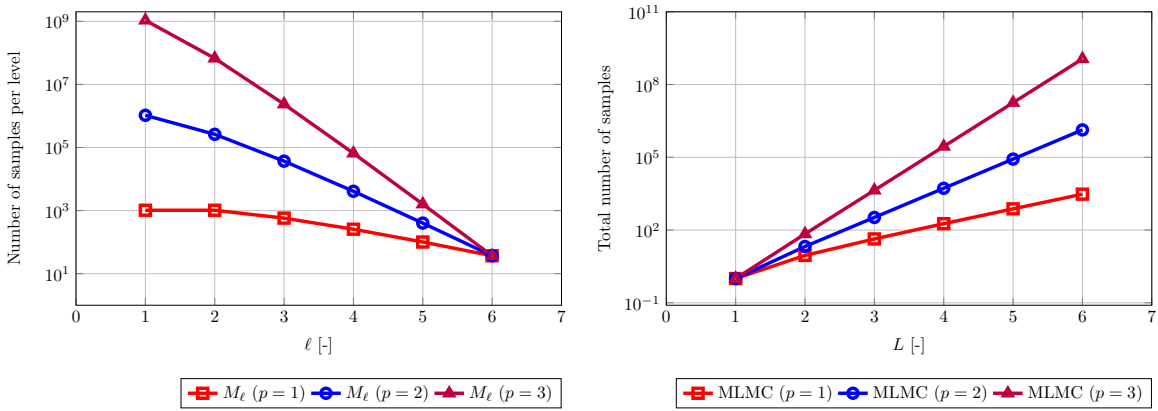
\begin{figure}[!htbp]
    \begin{subfigure}[b]{0.45\textwidth} 
          \resizebox{\textwidth}{!}{\pgfplotsset{
  log x ticks with fixed point/.style={
      xticklabel={
        \pgfkeys{/pgf/fpu=true}
        \pgfmathparse{exp(\tick)}%
        \pgfmathprintnumber[fixed  zerofill, precision=2]{\pgfmathresult}
        \pgfkeys{/pgf/fpu=false}
      }
  }
}

\begin{tikzpicture}

\begin{axis}[%
width=3.875in,
height=2.36in,
at={(2.6in,1.099in)},
scale only axis,
xmin=0,
xmax=7,
xminorticks=true,
xlabel = {$\ell$ [-]},
ylabel = {Number of samples per level},
ymode=log,
ymin=1,
ymax=2e9,
yminorticks=true,
axis background/.style={fill=white},
title style={font=\bfseries},
xmajorgrids,
ymajorgrids,
legend columns=3,
legend style={
        at={(1,-0.3)},
        anchor=east}
]

\addplot [color=red, mark=square,mark size=2.9pt, line width=2.0pt]
  table[row sep=crcr]{%
1   1025\\                    
2   1025\\
3   577\\
4   257\\
5   101\\
6   37\\
 };
\addlegendentry{$M_\ell$ ($p=1$)}

\addplot [color=blue, mark=o,mark size=2.9pt, line width=2.0pt]
table[row sep=crcr]{%
1   1048577\\                    
2   262145\\
3   36865\\
4   4097\\
5   401\\
6   37\\
 };
\addlegendentry{$M_\ell$ ($p=2$)}

\addplot [color=purple, mark=triangle,mark size=2.9pt, line width=2.0pt]
table[row sep=crcr]{%
1   1073741831\\
2   67108865\\
3   2359297\\
4   65537\\
5   1601\\
6   37\\
};
\addlegendentry{$M_\ell$ ($p=3$)}

\end{axis}
\end{tikzpicture}%

\pgfplotsset{
    every mark/.append style={solid},
}}
    \end{subfigure}
    \begin{subfigure}[b]{0.45\textwidth} 
          \resizebox{\textwidth}{!}{\pgfplotsset{
  log x ticks with fixed point/.style={
      xticklabel={
        \pgfkeys{/pgf/fpu=true}
        \pgfmathparse{exp(\tick)}%
        \pgfmathprintnumber[fixed  zerofill, precision=2]{\pgfmathresult}
        \pgfkeys{/pgf/fpu=false}
      }
  }
}

\begin{tikzpicture}

\begin{axis}[%
width=3.875in,
height=2.36in,
at={(2.6in,1.099in)},
scale only axis,
xmin=0  ,
xmax=7,
xminorticks=true,
xlabel = {$L$ [-]},
ylabel = {Total number of samples},
ymode=log,
ymin=0,
ymax=1e11,
yminorticks=true,
axis background/.style={fill=white},
title style={font=\bfseries},
xmajorgrids,
ymajorgrids,
legend columns=3,
legend style={
        at={(1,-0.3)},
        anchor=east}
]

\addplot [color=red, mark=square, mark size=2.9pt, mark options=solid, style = solid, line width=2.0pt]
  table[row sep=crcr]{%
1   1\\
2   9\\
3   43\\
4   183\\
5   749\\
6   3022\\
 };
\addlegendentry{MLMC ($p=1$)}

\addplot [color=blue, mark=o, mark size=2.9pt, mark options=solid, style = solid, line width=2.0pt]
  table[row sep=crcr]{%
1   1\\
2   21\\
3   331\\
4   5283\\
5   84509\\
6   1352122\\
 };
\addlegendentry{MLMC ($p=2$)}

\addplot [color=purple, mark=triangle, mark size=2.9pt, mark options=solid, style = solid, line width=2.0pt]
  table[row sep=crcr]{%
1   1\\
2   69\\
3   4363\\
4   279123\\
5   17863709\\
6   1143277162\\
 };
\addlegendentry{MLMC ($p=3$)}

\end{axis}
\end{tikzpicture}%

\pgfplotsset{
    every mark/.append style={solid},
}}
    \end{subfigure}
    \caption{(Left) Number of samples per level $\{M_\ell\}_{\ell=1}^6$ needed to compute the MLMC-VE estimator $\mlmc{[u]}=\sum_{\ell=1}^L \mc{M_\ell}[w_\ell]$ with $L=6$ according to~\eqref{eq:Mell_practical} with $\varepsilon=1e-10$. (Right) Total number of samples $\sum_{\ell=1}^LM_\ell$ needed to compute the MLMC estimator $\mlmc[u]$ for increasing $L=2,\ldots,6$. Both figures are in semilog scale.}
    \label{fig:NumberSamples}
\end{figure}

For the orders 
$p=1,2$, we compute the MLMC-VE estimator~$\mlmc[u]$ 
for increasing values of $L$ ($L=1,\ldots,5$ for $p=1$ and $L=1,\ldots,4$ for $p=2$, respectively). Since the analytical solution $\E[u]$ is unknown, for each $p$, we consider, as the reference solution, 
the MLMC-VE estimator with the maximum 
level $L+1$ (i.e., maximum level~$6$ 
for $p=1$ and maximum level~$5$ 
for $p=2$).
In Figure~\ref{fig:errors-mlmch} (left), we plot in loglog scale the $H^1(D)$ norm of the errors of the MC-VE method and MLMC-VE method versus $h_L$. The theoretical estimates obtained in Sections~\ref{sec:mcvem} and~\ref{sec:MLMC-VE} are confirmed (see Proposition~\ref{prop:mcvem} and Remark~\ref{rem:mc_Mopt} for the MC-VE method, and Theorem~\ref{thm:mlmc_Mopt} for the~MLMC-VE method). Indeed, the plots clearly show that both the MC-VE and MLMC-VE errors converge with order~$\mathcal O(h_L^p)$:
\[
\|\E[u]-\mc{M}[\Pi^\nabla_{h,p} u_h]\|_{H^1(D)}=\mathcal O(h_{L}^p), \qquad
\|\E[u]-\mlmc[u]\|_{H^1(D)}=\mathcal O(h_L^p).
\]
%
Figure~\ref{fig:errors-mlmch} (right) reports the MLMC-VE errors versus the maximum level $L$ in semilog scale, for~$L=1,\ldots,6$. 
The results shows exponential decay in~$L$ with approximate slope of $0.8$ for~$p=1$ and $1.5$ for $p=2$.
In the same setting, to assess the relative accuracy of MLMC-VE and MC-VE, we fix several levels of computational complexity and plot in loglog scale the corresponding errors for each method in Figure~\ref{fig:errors-mlmch-vs-L-compl}. We recall that, with the choices and the notation of Section~\ref{sec:implementation_choices}, the computational complexity~${\mathbb{C}}_{MLMC}(L_{MLMC})$ for MLCM with maximum level~$L_{MLCM}$ is, up to a constant,~$\sum_{\ell=1}^{L_{MLMC}}M_{\ell}N_{\ell}$, and the computational complexity~${\mathbb{C}}_{MC}(L_{MC})$ for MC at level~$L_{MC}$ is~$MN_{L_{MC}}$.
For a fixed value of~${\mathbb{C}}_{MLMC}$, the accuracy of the MLMC-VE method improves as $p$ increases. In contrast, for a fixed~${\mathbb{C}}_{MC}$, the MC-VE method produces solutions with a level of accuracy independent of~$p$. More importantly, at a given computational complexity, MLMC-VE consistently achieves higher accuracy, yielding errors roughly one order of magnitude smaller than those of MC-VE.

Finally, in Figure~\ref{fig:Qoi-errors-mlmch} (left), we plot in loglog scale the MLMC-VE error on $\E[\Q(u)]$ for the choice $\Q(u)=\int_D u\,\dx$. The number of samples~$M_\ell$ per level  is fixed according to~\eqref{eq:Mell_practical}, 
and the predicted rate of convergence 
\[
|\E[\Q(u)]-\mlmc[\Q(u)]|=\mathcal O(h_L^{2p}).
\]
is confirmed. The same errors are plotted in loglog scale versus the computational complexity in Figure~\ref{fig:Qoi-errors-mlmch} (right), where it is again evident that MLMC-VE consistently achieves higher accuracy than MC-VE.

\begin{figure}[!htbp]
    \begin{center}
    \begin{subfigure}[b]{0.45\textwidth} 
          \resizebox{\textwidth}{!}{\pgfplotsset{
  log x ticks with fixed point/.style={
      xticklabel={
        \pgfkeys{/pgf/fpu=true}
        \pgfmathparse{exp(\tick)}%
        \pgfmathprintnumber[fixed  zerofill, precision=2]{\pgfmathresult}
        \pgfkeys{/pgf/fpu=false}
      }
  }
}

\begin{tikzpicture}

\begin{axis}[%
width=3.875in,
height=2.36in,
at={(2.6in,1.099in)},
scale only axis,
xmode=log,
xmin=0.01,
xmax=1,
xminorticks=true,
xlabel = {$h_L$ [-]},
ylabel = {Errors on $\E[u]$},
ymode=log,
ymin=1e-5,
ymax=1e-1,
yminorticks=true,
axis background/.style={fill=white},
title style={font=\bfseries},
xmajorgrids,
ymajorgrids,
legend columns=3,
legend style={
        at={(1,-0.3)},
        anchor=east}
]


\addplot [color=orange, mark=square, mark size=2.9pt, mark options=solid, style = solid, line width=2.0pt]
  table[row sep=crcr]{%
3.535534e-01   0.0180035725514754\\
1.767767e-01   0.00809842419951163\\
8.838835e-02   0.00349574600016957\\
4.419417e-02   0.00149465798647225\\
2.209709e-02   0.000663505872640797\\
 };
\addlegendentry{MLMC ($p=1$)}
                           
\addplot [color=magenta, mark=o, mark size=2.9pt, mark options=solid, style = solid, line width=2.0pt]
  table[row sep=crcr]{%
3.535534e-01  0.00427289635214014\\
1.767767e-01  0.00107419559081763\\
8.838835e-02  0.000251349867446481\\
4.419417e-02  6.86298740958629e-05\\
 };
\addlegendentry{MLMC ($p=2$)}


\addplot [color=blue, mark=square, mark size=2.9pt, mark options=solid, style = dashed, line width=2.0pt]
  table[row sep=crcr]{%
3.535534e-01   0.0174672648643881\\
1.767767e-01   0.00799412093192411\\
8.838835e-02   0.00331383391423862\\
4.419417e-02   0.00151485867456923\\
2.209709e-02   0.000668469505026143\\
 };
\addlegendentry{MC ($p=1$)}

\addplot [color=teal, mark=o, mark size=2.9pt, mark options=solid, style = dashed, line width=2.0pt]
  table[row sep=crcr]{%
3.535534e-01   0.00376765379416038\\
1.767767e-01   0.00095800592256433\\
8.838835e-02   0.000245370897454095 \\
 };
\addlegendentry{MC ($p=2$)}

\addplot [color=black, mark=square, mark size=2.9pt, mark options=solid, style = dashdotdotted, line width=2.0pt]
  table[row sep=crcr]{%
3.535534e-01   3.535534e-02\\
1.767767e-01   1.767767e-02\\
8.838835e-02   8.838835e-03\\
4.419417e-02   4.419417e-03\\
2.209709e-02   2.209709e-03\\
 };
\addlegendentry{$h_L$}
                           
\addplot [color=black, mark=o, mark size=2.9pt, mark options=solid, style = dashdotdotted, line width=2.0pt]
  table[row sep=crcr]{%
3.535534e-01  0.0125000000628095 \\
1.767767e-01  0.00312500001570237 \\
8.838835e-02  0.000781250003925593\\
4.419417e-02  0.000195312500981398\\
 };
\addlegendentry{$h_L^2$}

\end{axis}
\end{tikzpicture}%
}
    \end{subfigure}%
    \begin{subfigure}[b]{0.45\textwidth} 
          \resizebox{\textwidth}{!}{\pgfplotsset{
  log x ticks with fixed point/.style={
      xticklabel={
        \pgfkeys{/pgf/fpu=true}
        \pgfmathparse{exp(\tick)}%
        \pgfmathprintnumber[fixed  zerofill, precision=2]{\pgfmathresult}
        \pgfkeys{/pgf/fpu=false}
      }
  }
}

\begin{tikzpicture}

\begin{axis}[%
width=3.875in,
height=2.36in,
at={(2.6in,1.099in)},
scale only axis,
xmin=0,
xmax=6,
xminorticks=true,
xlabel = {$L$ [-]},
ylabel = {Errors on $\E[u]$},
ymode=log,
ymin=1e-5,
ymax=1e-1,
yminorticks=true,
axis background/.style={fill=white},
title style={font=\bfseries},
xmajorgrids,
ymajorgrids,
legend columns=3,
legend style={
        at={(1,-0.3)},
        anchor=east}
]


\addplot [color=orange, mark=square, mark size=2.9pt, mark options=solid, style = solid, line width=2.0pt]
  table[row sep=crcr]{%
1   0.0180035725514754\\
2   0.00809842419951163\\
3   0.00349574600016957\\
4   0.00149465798647225\\
5   0.000663505872640797\\
 };
\addlegendentry{MLMC ($p=1$)}
                           
\addplot [color=magenta, mark=o, mark size=2.9pt, mark options=solid, style = solid, line width=2.0pt]
  table[row sep=crcr]{%
1  0.00427289635214014\\
2  0.00107419559081763\\
3  0.000251349867446481\\
4  6.86298740958629e-05\\
 };
\addlegendentry{MLMC ($p=2$)}


\addplot [color=black, mark=square, mark size=2.9pt, mark options=solid, style = dashdotdotted, line width=2.0pt]
  table[row sep=crcr]{%
1   0.0449328964117222\\
2   0.0201896517994655\\
3   0.00907179532894125\\
4   0.00407622039783662\\
5   0.00183156388887342\\
 };
\addlegendentry{$L^{-0.8}$}
                           
\addplot [color=black, mark=o, mark size=2.9pt, mark options=solid, style = dashdotdotted, line width=2.0pt]
  table[row sep=crcr]{%
1   0.0022313016014843\\
2   0.00049787068367864\\
3   0.000111089965382423\\
4   2.47875217666636e-05\\
 };
\addlegendentry{$L^{-1.5}$}

\end{axis}
\end{tikzpicture}%
}
    \end{subfigure}%
    \end{center}
    \caption{
    (Left) MC-VE and MLMC-VE errors on $\E[u]$ in the $H^1(D)$ norm, plotted versus $h_{L}$ in loglog scale. (Right) MLMC-VE errors on $\E[u]$ in the $H^1(D)$ norm, plotted versus $L$ in semilog scale.
    }
    \label{fig:errors-mlmch}
\end{figure}
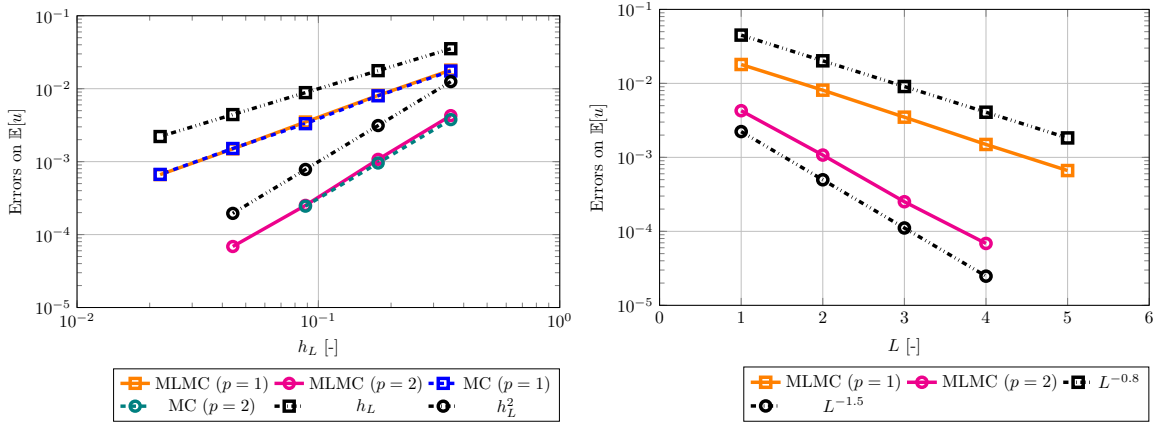

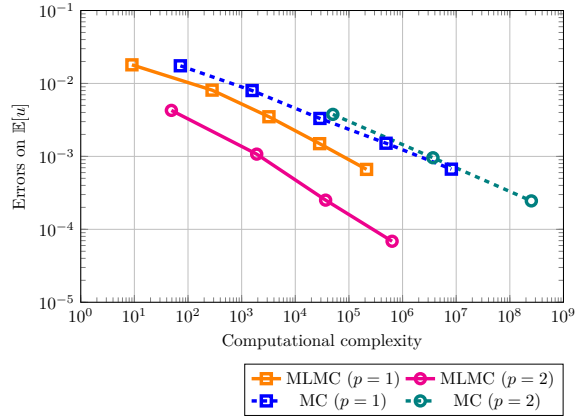
\begin{figure}[!htbp]
\begin{center}
    \begin{subfigure}[b]{0.45\textwidth} 
          \resizebox{\textwidth}{!}{\pgfplotsset{
  log x ticks with fixed point/.style={
      xticklabel={
        \pgfkeys{/pgf/fpu=true}
        \pgfmathparse{exp(\tick)}%
        \pgfmathprintnumber[fixed  zerofill, precision=2]{\pgfmathresult}
        \pgfkeys{/pgf/fpu=false}
      }
  }
}

\begin{tikzpicture}

\begin{axis}[%
width=3.875in,
height=2.36in,
at={(2.6in,1.099in)},
scale only axis,
ymode=log,
ymin=1e-5,
ymax=1e-1,
yminorticks=true,
ylabel = {Errors on $\mathbb{E}[u]$},
xlabel = {Computational complexity},
xmode=log,
xmin=1,
xmax=1e9,
xminorticks=true,
axis background/.style={fill=white},
title style={font=\bfseries},
xmajorgrids,
ymajorgrids,
legend columns=2,
legend style={
        at={(1,-0.3)},
        anchor=east}
]


\addplot [color=orange, mark=square, mark size=2.9pt, mark options=solid, style = solid, line width=2.0pt]
  table[row sep=crcr]{%
9  0.0180035725514754\\
281  0.00809842419951163\\
3227    0.00349574600016957\\
28423  0.00149465798647225\\
213181  0.000663505872640797\\
 };
\addlegendentry{MLMC ($p=1$)}                    

\addplot [color=magenta, mark=o, mark size=2.9pt, mark options=solid, style = solid, line width=2.0pt]
  table[row sep=crcr]{%
49  0.00427289635214014\\
1909  0.00107419559081763\\
36779  0.000251349867446481\\
638147  6.86298740958629e-05\\ 
};
\addlegendentry{MLMC ($p=2$)}


\addplot [color=blue, mark=square,mark size=2.9pt, style = dashed, mark options=solid, line width=2.0pt]
  table[row sep=crcr]{%
  72      0.0174672648643881\\
  1568     0.00799412093192411 \\
  28800    0.00331383391423862\\
  492032    0.00151485867456923\\
  8128512   0.000668469505026143\\
  };
\addlegendentry{MC ($p=1$)} 

\addplot [color=teal, mark=o,mark size=2.9pt, style = dashed, mark options=solid, line width=2.0pt]
  table[row sep=crcr]{%
  50176  0.00376765379416038\\
  3686400  0.00095800592256433\\
  251920384  0.000245370897454095\\
 }; 
\addlegendentry{MC ($p=2$)}

\end{axis}
\end{tikzpicture}%
}
    \end{subfigure}%
    \caption{
    MLMC-VE and MC-VE errors on $\E[u]$ in the $H^1(D)$ norm versus the computational complexity in loglog scale. 
    }
    \label{fig:errors-mlmch-vs-L-compl}
\end{center}
\end{figure}
\begin{figure}[!htbp]
\begin{center}
    \begin{subfigure}[b]{0.46\textwidth} 
          \resizebox{\textwidth}{!}{\pgfplotsset{
  log x ticks with fixed point/.style={
      xticklabel={
        \pgfkeys{/pgf/fpu=true}
        \pgfmathparse{exp(\tick)}%
        \pgfmathprintnumber[fixed  zerofill, precision=2]{\pgfmathresult}
        \pgfkeys{/pgf/fpu=false}
      }
  }
}

\begin{tikzpicture}

\begin{axis}[%
width=3.875in,
height=2.36in,
at={(2.6in,1.099in)},
scale only axis,
xmode=log,
xmin=0.01,
xmax=1,
xminorticks=true,
xlabel = {$h_L$ [-]},
ylabel = {Errors on $\E[\mathcal Q(u)]$},
ymode=log,
ymin=1e-8,
ymax=1e-1,
yminorticks=true,
axis background/.style={fill=white},
title style={font=\bfseries},
xmajorgrids,
ymajorgrids,
legend columns=3,
legend style={
        at={(1,-0.3)},
        anchor=east}
]


\addplot [color=orange, mark=square, mark size=2.9pt, mark options=solid, style = solid, line width=2.0pt]
  table[row sep=crcr]{%
3.535534e-01   0.00303291114927012\\
1.767767e-01   0.0011287009447539\\
8.838835e-02   0.000283653108032327\\
4.419417e-02   5.68111870706822e-05\\
 };                     
\addlegendentry{MLMC ($p=1$)}

\addplot [color=magenta, mark=o, mark size=2.9pt, mark options=solid, style = solid, line width=2.0pt]
  table[row sep=crcr]{%
3.535534e-01   0.000454773874067745\\
1.767767e-01   1.12311520423687e-05\\
8.838835e-02   4.42394662222845e-08 \\
 };                     
\addlegendentry{MLMC ($p=2$)}

\addplot [color=blue, mark=square, mark size=2.9pt, mark options=solid, style = dashed, line width=2.0pt]
  table[row sep=crcr]{%
3.535534e-01   0.00290675961905482\\
1.767767e-01   0.00103537253622074\\
8.838835e-02   0.000282130865441006\\
4.419417e-02   5.84599853465703e-05\\
 };
\addlegendentry{MC ($p=1$)}

\addplot [color=teal, mark=o, mark size=2.9pt, mark options=solid, style = dashed, line width=2.0pt]
  table[row sep=crcr]{%
3.535534e-01   0.000178886328696763\\
1.767767e-01   1.76225525332632e-05\\
8.838835e-02   1.34199065677815e-06\\
 };
\addlegendentry{MC ($p=2$)}

\addplot [color=black, mark=square, mark size=2.9pt, mark options=solid, style = dashdotdotted, line width=2.0pt]
  table[row sep=crcr]{%
3.535534e-01  0.0125000000628095 \\
1.767767e-01  0.00312500001570237 \\
8.838835e-02  0.000781250003925593\\
4.419417e-02  0.000195312500981398\\
 };
\addlegendentry{$h_L^2$}
                           
\addplot [color=black, mark=o, mark size=2.9pt, mark options=solid, style = dashdotdotted, line width=2.0pt]
  table[row sep=crcr]{%
3.535534e-01  0.00156250001570237 \\
1.767767e-01  0.0000976562509813983 \\
8.838835e-02  6.10351568633739e-06\\
 };
\addlegendentry{$h_L^4$}

\end{axis}
\end{tikzpicture}%
}
    \end{subfigure}%
    \begin{subfigure}[b]{0.46\textwidth} 
          \resizebox{\textwidth}{!}{\pgfplotsset{
  log x ticks with fixed point/.style={
      xticklabel={
        \pgfkeys{/pgf/fpu=true}
        \pgfmathparse{exp(\tick)}%
        \pgfmathprintnumber[fixed  zerofill, precision=2]{\pgfmathresult}
        \pgfkeys{/pgf/fpu=false}
      }
  }
}

\begin{tikzpicture}

\begin{axis}[%
width=3.875in,
height=2.36in,
at={(2.6in,1.099in)},
scale only axis,
xmode=log,
xmin=1,
xmax=1e12,
xminorticks=true,
xlabel = {Computational complexity},
ylabel = {Errors on $\E[\mathcal Q(u)]$},
ymode=log,
ymin=1e-8,
ymax=1e-2,
yminorticks=true,
axis background/.style={fill=white},
title style={font=\bfseries},
xmajorgrids,
ymajorgrids,
legend columns=3,
legend style={
        at={(1,-0.3)},
        anchor=east}
]


\addplot [color=orange, mark=square, mark size=2.9pt, mark options=solid, style = solid, line width=2.0pt]
  table[row sep=crcr]{%
9   0.00303291114927012\\
389   0.0011287009447539\\
7739   0.000283653108032327\\
136051   5.68111870706822e-05\\
 };                     
\addlegendentry{MLMC ($p=1$)}

\addplot [color=magenta, mark=o, mark size=2.9pt, mark options=solid, style = solid, line width=2.0pt]
  table[row sep=crcr]{%
49   0.000454773874067745\\
13669   1.12311520423687e-05\\
3451499   4.42394662222845e-08 \\
 };                     
\addlegendentry{MLMC ($p=2$)}

\addplot [color=blue, mark=square, mark size=2.9pt, mark options=solid, style = dashed, line width=2.0pt]
  table[row sep=crcr]{%
576   0.00290675961905482\\
50176   0.00103537253622074\\
3686400   0.000282130865441006\\
251920384   5.84599853465703e-05\\
 };
\addlegendentry{MC ($p=1$)}

\addplot [color=teal, mark=o, mark size=2.9pt, mark options=solid, style = dashed, line width=2.0pt]
  table[row sep=crcr]{%
200704   0.000178886328696763\\
235929600  1.76225525332632e-05\\
257966468411   1.34199065677815e-06\\
 };
\addlegendentry{MC ($p=2$)}

\end{axis}
\end{tikzpicture}%
}
    \end{subfigure}%
    \caption{
    (Left) MLMC-VEM errors on $\E[\Q(u)]$ plotted versus $h_{L}$ in loglog scale. (Right) MLMC-VEM errors on $\E[\Q(u)]$ versus the computational complexity in loglog scale. $\Q(u)=\int_D u\, \dx$.
    }
    \label{fig:Qoi-errors-mlmch}
\end{center}
\end{figure}
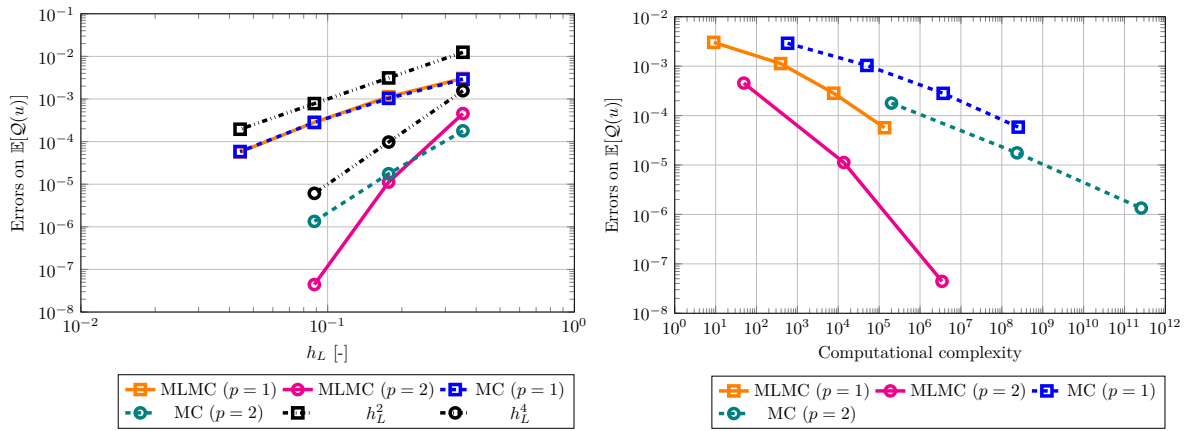

\FloatBarrier

\subsection{Validation test}
\label{sec:emilia}

In this section, we present some numerical results to demonstrate the practical capabilities of the proposed scheme. We consider problem~\eqref{eq:strong_pde} defined over the rectangular domain $D=(0,4)\times(0,1)$, which is split into seven non-overlapping regions $\{D_r\}_{r=1}^7$ (see Figure~\ref{fig:regions}) modeling rock strata in the subsurface. We take the forcing term constant and
equal to~$1$. The random diffusion coefficient is defined as 
$a(\omega,x)=\sum_{r=1}^7\chi_{D_r}(x)Y_r(\omega)$,
where~$\{Y_r\}_{r=1}^7$ are uniformly distributed random variables. In particular,  \[
a(\omega,x)|_{D_r}=Y_r(\omega)\in\mathbb{R}\qquad \text{for all $r=1,\ldots,7$},
\]
so that the realizations of the random diffusion coefficient are piecewise constant over the regions, with each local coefficient following a uniform distribution. This example goes beyond the theoretical framework considered in the previous sections, as the solution of problem~\eqref{eq:strong_pde} with a discontinuous diffusion coefficient does not satisfy the regularity assumption $u\in L^p(\Omega,H^{2}(D))$ (see Proposition~\ref{prop:mcvem} and Theorem~\ref{thm:mlmc_u_conv}).
\begin{figure}[!htbp]
\begin{center}
   \begin{subfigure}[b]{0.6\textwidth} 
          \includegraphics[width=\textwidth, height=5cm]{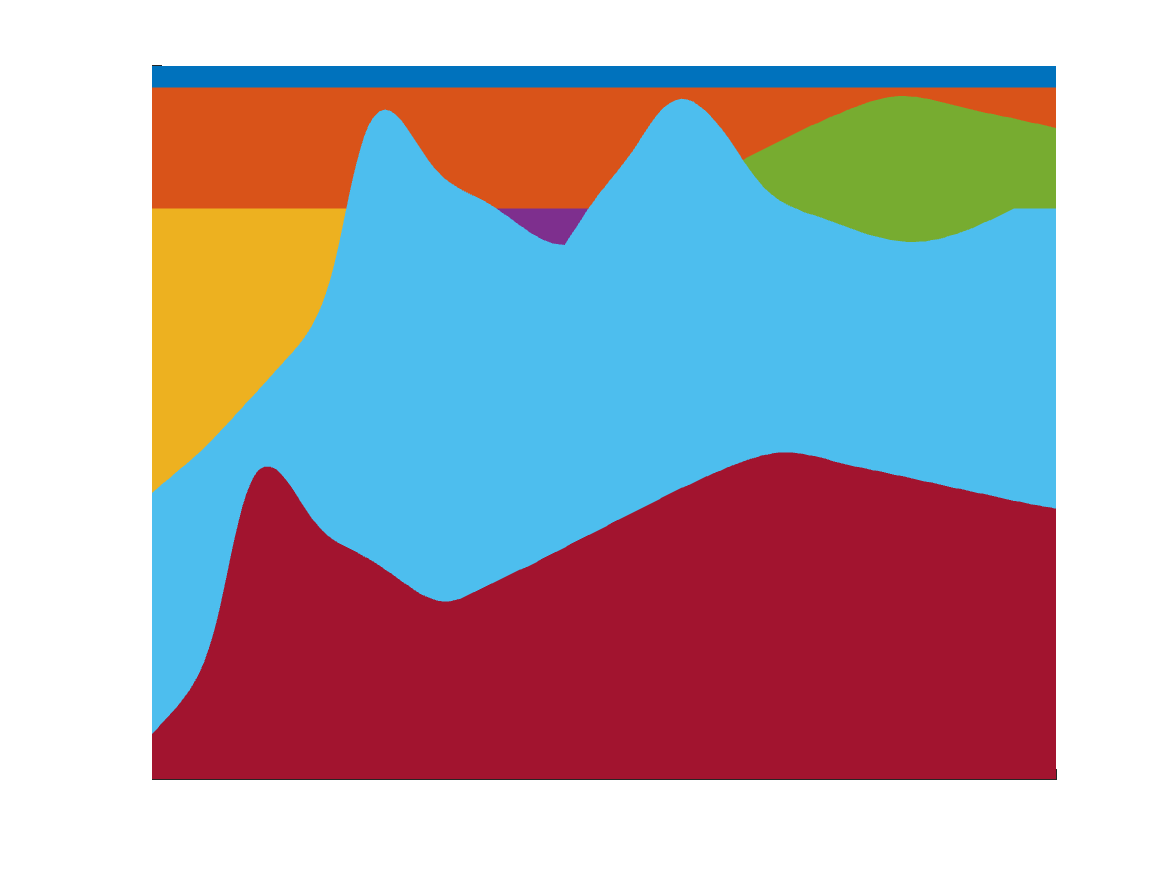}
    \end{subfigure}%
    \caption{Subdivision of the rectangular domain into subregions, as considered in Section~\ref{sec:emilia}.}
    \label{fig:regions}
\end{center}
\end{figure}
\begin{figure}[!htbp]
   \begin{subfigure}[b]{0.5\textwidth}
          \includegraphics[width=\textwidth,height=5cm]{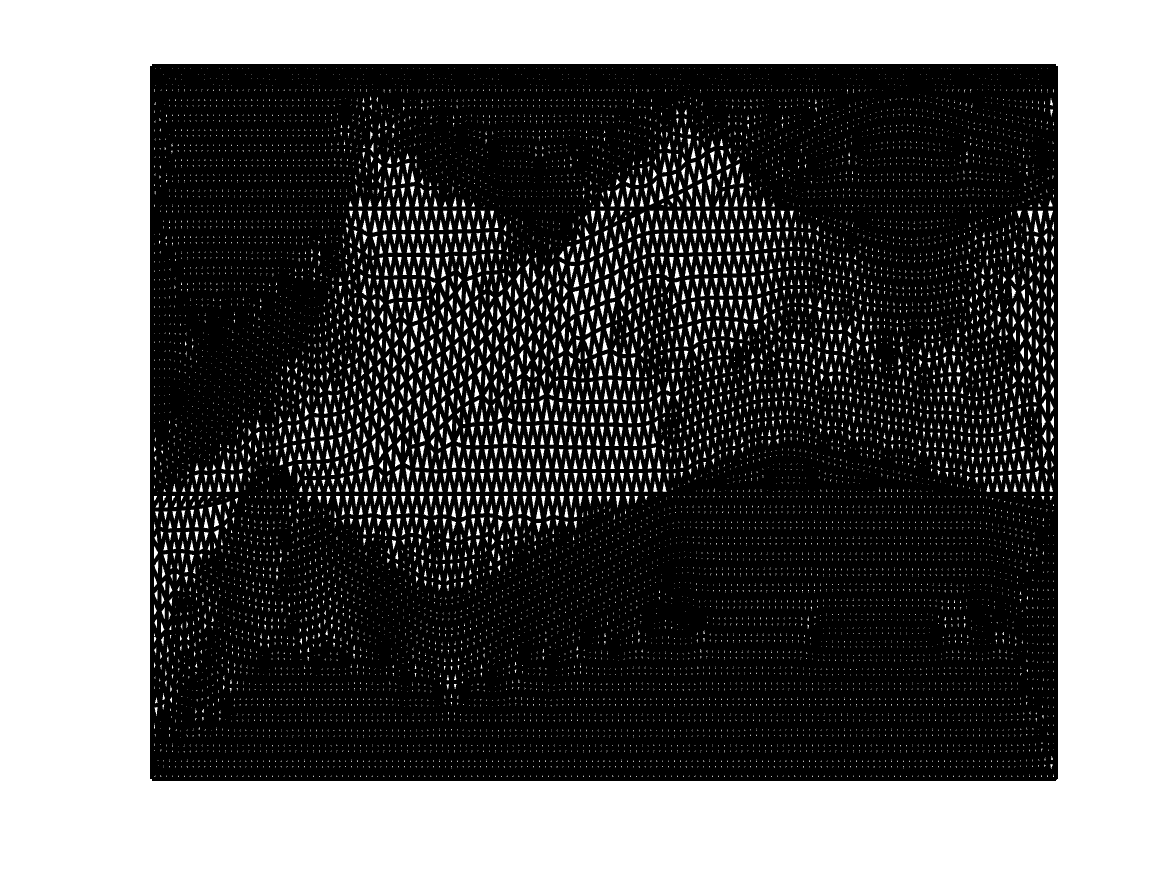}
    \end{subfigure}%
    \begin{subfigure}[b]{0.5\textwidth}
          \includegraphics[width=\textwidth,height=5cm]{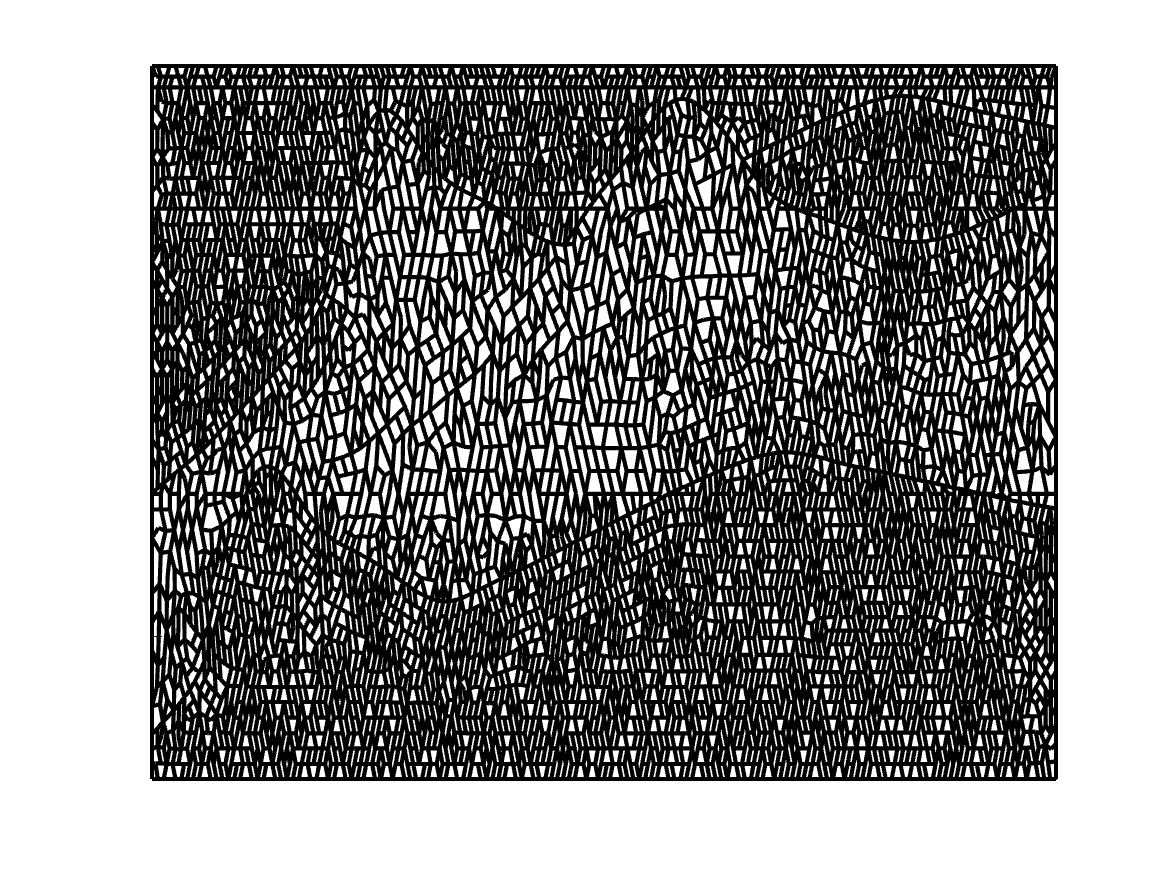}
    \end{subfigure}%
    \\
    \begin{subfigure}[b]{0.5\textwidth}
          \includegraphics[width=\textwidth,height=5cm]{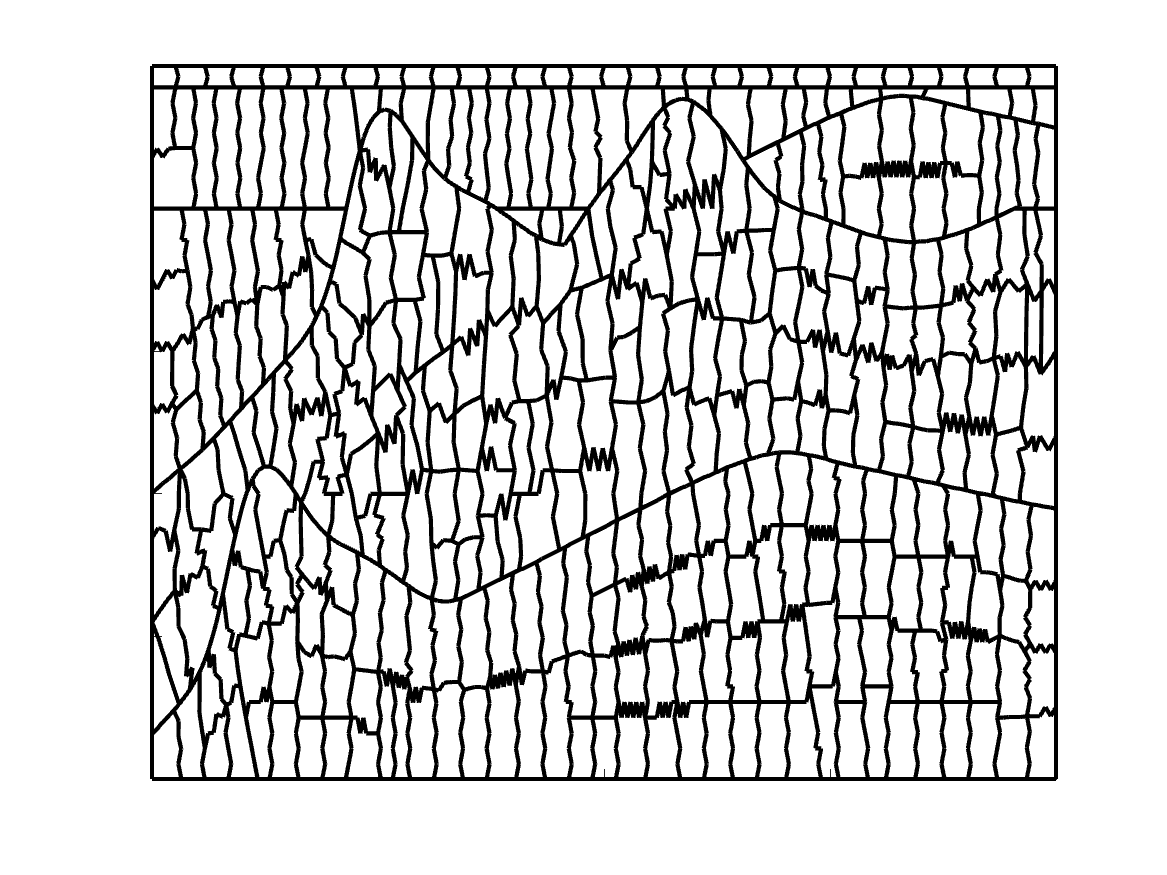}
    \end{subfigure}%
    \begin{subfigure}[b]{0.5\textwidth}
          \includegraphics[width=\textwidth,height=5cm]{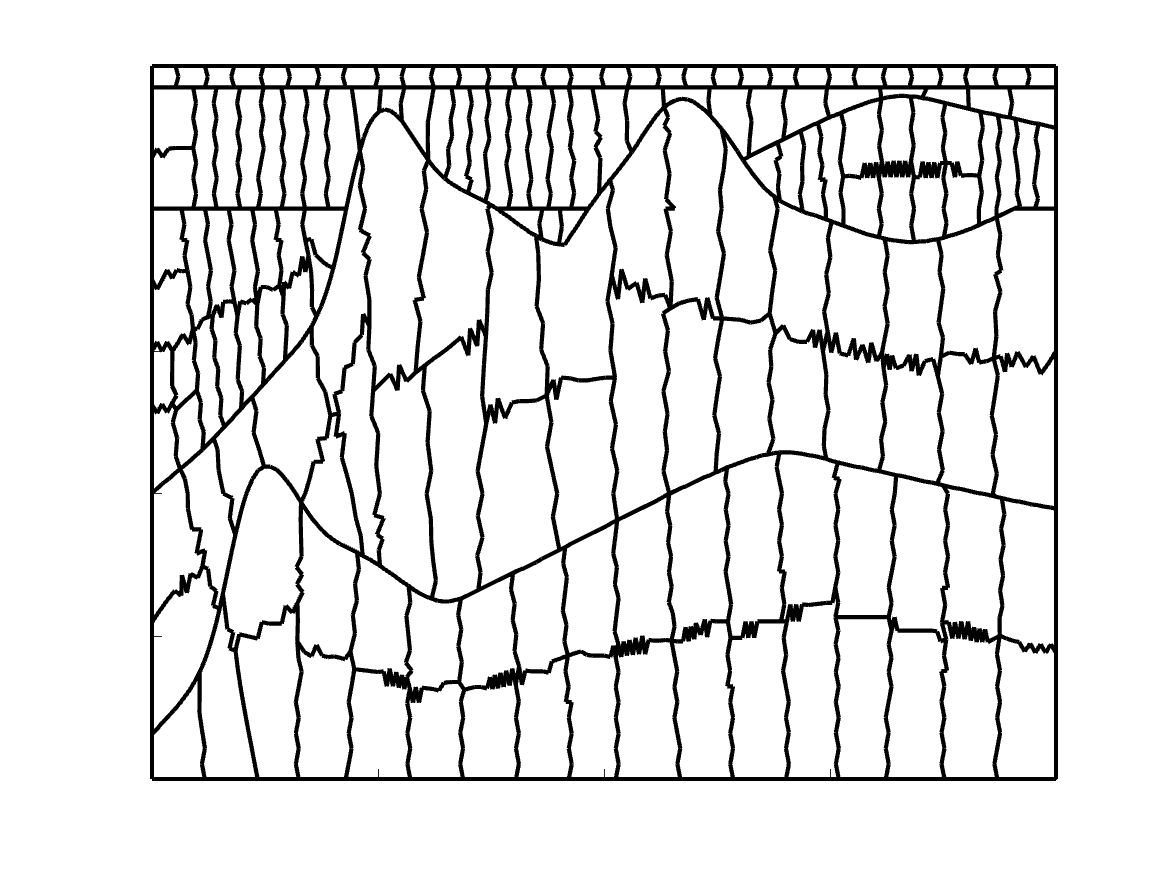}
    \end{subfigure}%
    \caption{Sequence of nested meshes used in the numerical tests of Section~\ref{sec:emilia}, obtained via a mesh-agglomeration strategy from a fine triangular grid shown in the top--right panel. 
    }
    \label{fig:MeshesEmilia}
\end{figure}
The MLMC-VE method is defined on the sequence of four nested polygonal meshes $\{\T_\ell\}_{\ell= 0}^3$ represented in Figure~\ref{fig:MeshesEmilia}. The sequence satisfies~$h_\ell\simeq\frac{h_{\ell-1}}{2}$ and is constructed using a mesh-agglomeration strategy. Specifically, starting from the fine (triangular) mesh~$\ T_L=\T_3$, at each level $\ell=L-1, \ldots,  0$, we recursively group fine-level elements into coarse agglomerates so that each coarse cell is the union of a set of fine cells. For the numerical simulations, we use the algorithm proposed in \cite{ANTONIETTI2024, Antonietti2026} and implemented in \texttt{MAGNET}, an open-source Python library for mesh agglomeration in two and three dimensions based on Graph Neural Networks (see~\cite{antonietti2025magnet}).

\begin{figure}
    \centering
    \includegraphics[width=0.5\textwidth]{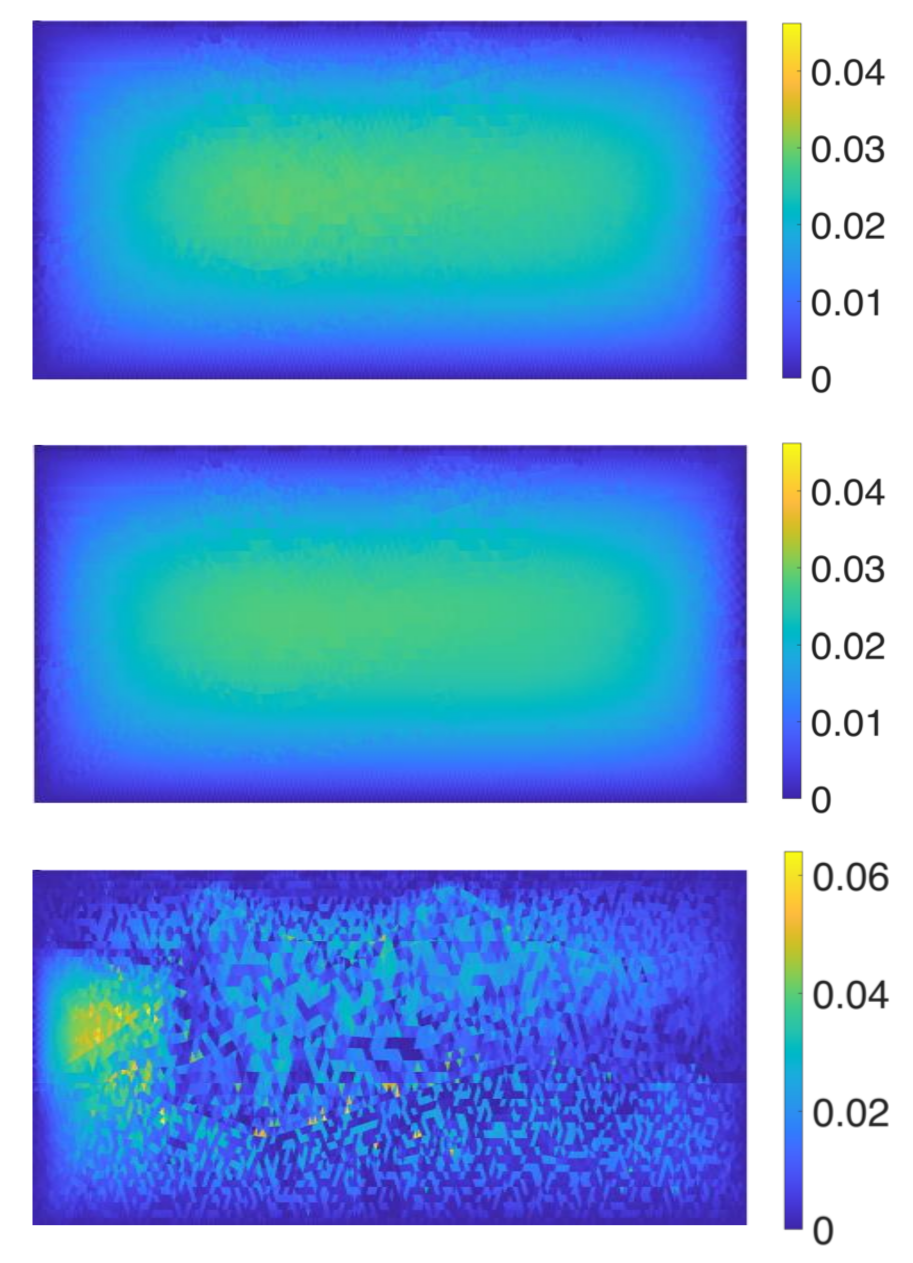}
    \caption{(Top): MLMC-VE approximation of $\E[u]$ for $a(\omega,x)=\sum_{r=1}^7\chi_{D_r}(x)Y_r(\omega)$, where $Y_r\sim\mathcal U([1,10])$ for all $r=1,\ldots,7$.
    (Middle): MC approximation of $\E[u]$.
    (Bottom): Absolute value of the difference of the two approximations divided by the maximum norm of the MLMC-VE approximation.
    }
    \label{fig:aggEmilia1}
\end{figure}


We first consider the case $Y_r\sim\mathcal U([1,10])$ for all $r=1,\ldots,7$, modeling the setting where the diffusivity of all regions is comparable (ranging between 1 and 10), namely, all regions are composed of the same material.
In Figure~\ref{fig:aggEmilia1}, we compare the MLMC estimator (top) with the MC estimator (middle) of $\E[u]$. The pointwise absolute difference is plotted in Figure~\ref{fig:aggEmilia1} (bottom), showing agreement between the two approximations up to the third decimal place. 
The same quantities are then computed for~$Y_r\sim\mathcal U([1,2])$ for all $r\neq 7$ and $Y_7\sim\mathcal U([1,100])$.
This configuration models a case in which the first six regions have comparable diffusivity (varying within the smaller interval~$[1,2]$), whereas the seventh region (the red one in Figure~\ref{fig:regions}) has a much larger diffusivity, ranging from ~$1$ to~$100$. The latter region can therefore be interpreted as a layer made of a different material.
The MLMC and the MC estimators, as well as their pointwise absolute difference, are reported in Figure~\ref{fig:aggEmilia2}. 
Again, we observe agreement between the two approximations to three decimal digits. Moreover, a comparison between  Figures~\ref{fig:aggEmilia1} and~\ref{fig:aggEmilia2} shows that the pointwise profile of the estimators is affected by the range of variation of the random variables $\{Y_r\}_{r=1}^7$, becoming flatter in regions characterized by higher diffusivity.
Finally, we consider the QoI defined as the average over the physical domain~$D$ and compare the MC-VE and MLMC-VE estimators of its expected value. We observe agreement to five decimal digits in both configurations: when the diffusivity of all regions is comparable, and when the first six regions have comparable diffusivity, whereas the seventh has a larger diffusivity.

\begin{figure}
    \centering
    \includegraphics[width=0.5\textwidth]{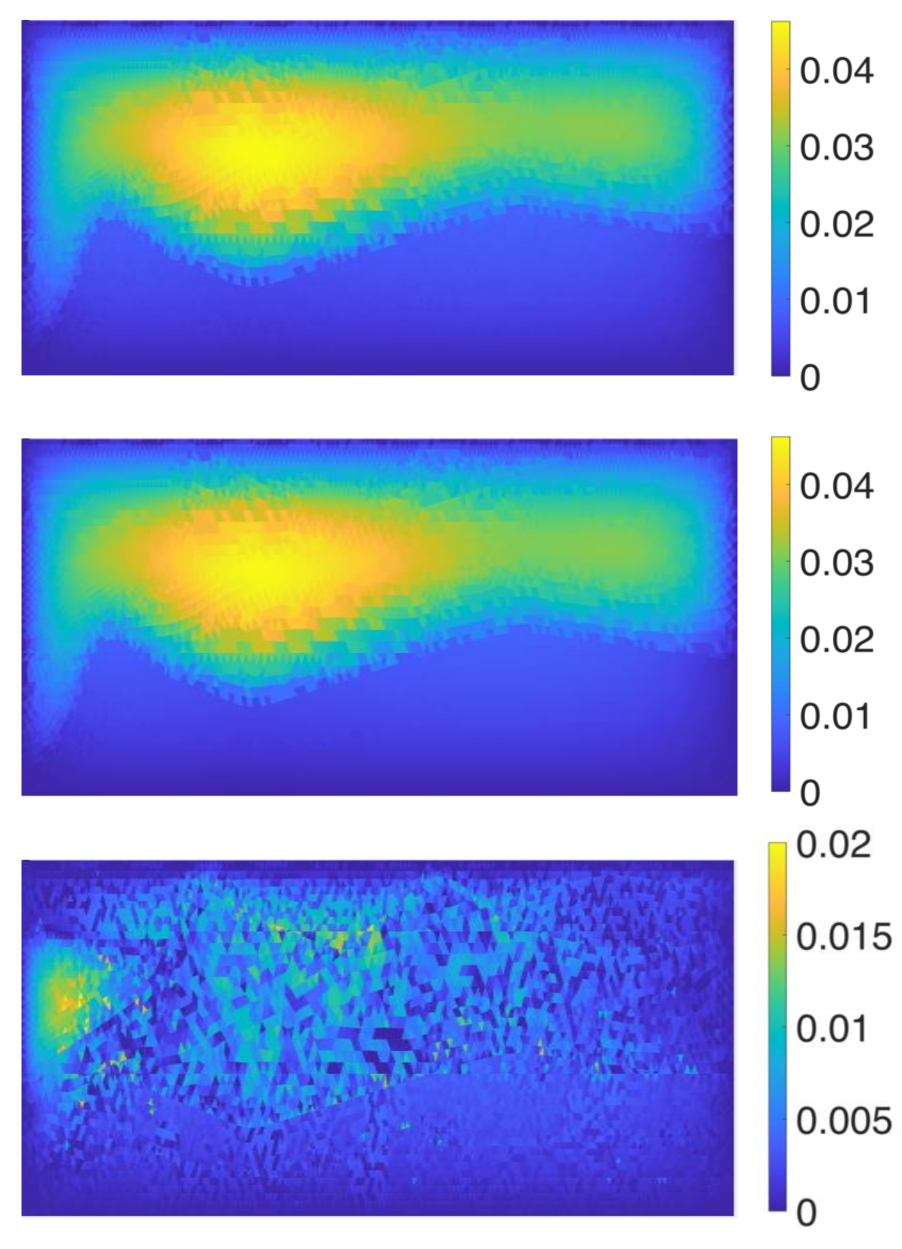}
    \caption{(Top): MLMC-VE approximation of $\E[u]$ for $a(\omega,x)=\sum_{r=1}^7\chi_{D_r}(x)Y_r(\omega)$, where $Y_r\sim\mathcal U([1,2])$ for all $r\neq 7$ and $Y_7\sim\mathcal U([1,100])$.
    (Middle): MC approximation of $\E[u]$.
   (Bottom): Absolute value of the difference of the two approximations divided by the maximum norm of the MLMC-VE approximation.
    }
    \label{fig:aggEmilia2}
\end{figure}

\FloatBarrier

\section{Conclusions}
\label{sec:conclusions}
In this work, we introduced a Monte Carlo Virtual Element method for stochastic elliptic partial differential equations with random diffusion coefficients and proved error estimates for the statistical approximation error of both the solution and relevant (linear) quantities of interest. We also developed and analyzed a Multilevel Monte Carlo Virtual Element method that integrates stochastic sampling with a hierarchy of VE spaces constructed from agglomerated meshes. We proved the convergence of the proposed Multilevel Monte Carlo Virtual Element method and demonstrated that the MLMC-VE strategy yields significant cost reductions compared with the standard Monte Carlo approach.  Numerical experiments demonstrate the theoretical bounds and the computational efficiency of the proposed approach. 
Possible future developments include extending the proposed analysis to the $p$-version of the Virtual Element method
and stochastic multiphysics models with uncertain data. Another promising direction is the investigation of adaptive multilevel strategies to further reduce computational costs.

\section*{Acknowledgments}
The author sincerely thanks Matteo Caldana and Lorenzo Mascotto for their helpful discussions and valuable suggestions during the development of this work.

\bibliographystyle{abbrv}
\bibliography{biblio_new}
\end{document}